\documentclass{article}[8pt]

\usepackage{amssymb}
\usepackage{amsmath}
\usepackage{amsthm}
\usepackage{dsfont} 
\usepackage [applemac] {inputenc}
\usepackage{geometry} 
\usepackage{float}
\usepackage{color,graphics}
 
\usepackage[T1]{fontenc}
\usepackage{cite}    
\usepackage[francais,english]{babel}  
\usepackage[all]{xy}
\usepackage{graphicx}
\usepackage{bibentry}

\usepackage{tikz}

\begin{document}

\title{Symbolic approach and induction in the Heisenberg group}
\author{{\Large Jean-François Bertazzon \footnote{Adresse e-mail: bertazzon@iml.univ-mrs.fr}} \\ {\it \small Institut de Mathématiques de Luminy (UMR 6206),} \\  {\it \small Université de la Méditérranée,} \\ {\it \small Campus de Luminy, 13288 MARSEILLE Cedex 9, France} }

\newenvironment{abst}{%
\hspace{+2.5\parindent}\begin{minipage}{.85\linewidth} {\textbf{Abstract.}}}
{\end{minipage}\par}

\newtheorem{theor}{Theorem}
\renewcommand{\thetheor}{\Alph{theor}}

\newtheorem{theoreme}{Theorem}

\newtheorem{lemme}{Lemma}
\newtheorem{proposition}{Proposition}

\newcommand{\floor}[1]{{\left\lfloor #1 \right\rfloor}}



\maketitle

\begin{abst}\small{
We associate a homomorphism  in the Heisenberg group
to each hyperbolic unimodular automorphism of the free group on two generators.
We show that the first return-time of some flows in "good" sections, 
are conjugate to niltranslations, which have the property of being self-induced.
}
\end{abst}
\mbox{}
\vspace{6pt}
\\
We introduce an extension of the well known connection between
geometric systems and symbolic systems (broken lines, Rauzy fractals, etc ),
to the non commutative setting, i.e. Heisenberg group.
The symbolic objects considered are automorphisms of the free group on $m$ generators
$\mathbb F_m$.
If $G$ is an arbitrary group generated by $m$ generators, then there exists a surjective group homomorphism $\pi$ mapping $\mathbb F_m$ onto $G$.
The goal is then to translate the action of 
an automorphism $\sigma$ on $\mathbb F_m$, in an application on the group $G$, through the application $\pi$.
\vspace{6pt}
\\
Let $\sigma$ be a  \textbf{substitution} on $m$ letters,
i.e. a positive endomorphisms of the free group on $m$ generators $\mathbb F_m$.
Suppose there exists an unique infinite word
$\boldsymbol{u}=(u_n)_{n\in \mathbb N}$ such that
$\sigma(\boldsymbol{u})=\boldsymbol{u}$.
A natural construction exists, which associates to the substitution $\sigma$,
a sequence $(x_n)_n$ of elements of $G$, such that the $n$th term of the sequence,
is the projection by $\pi$ of the prefix $u_0 \dots u_n$.
(i.e. $ x_n = \pi (u_0 \dots u_n) $). This sequence of elements is called the \textbf{broken line}
in $ G $ associated to the substitution $\sigma$.
When $ G $ is the group $ \mathbb Z ^m $, the abelianisation of the free group on $ m $ generators,
under some assumptions on the substitution, the closure of a projection of the broken line
is a compact set of $ \mathbb R ^ {m-1} $, called a \textbf{Rauzy fractal}.
We can then measurably conjugate the symbolic dynamical system generated by $\boldsymbol{u}$,
with an exchange of pieces of this fractal.
There are many generalizations of this construction, especially for
free groups with more generators
(see \cite {PierreX}, \cite{PierreA} and \cite{MR1879664}).
\vspace {6pt}
\\
Another method is to translate the action of the substitution to 
the group $G$, in a way which is consistent with the morphism $ \pi $.
A topological group $ G $ with $ m $ generators will be called
\textbf{adapted for automorphisms of the free group $\mathbb F_m $}
endowed with its natural topology,
if the map $ \pi $ is continuous and if there is a continuous and surjective homomorphism
$ \mathfrak S $ : Aut $ (\mathbb F_m, \mathbb F_m) \longrightarrow $ Aut$(G, G) $
 such that for any $ \sigma \in$ Hom$ (\mathbb F_m, \mathbb F_m)$, the following diagram commutes :
$$
 \xymatrix{
    \mathbb F_m  \ar@{>}[r]^{\sigma} \ar@{>}[d]_{\pi} 
    & \mathbb F_m   \ar@{>}[d]^{\pi} \\
    G \ar@{>}[r]_{\mathfrak S _\sigma} 
    &G
  }
$$
We shall then say that $ \mathfrak S _ \sigma $ is the \textbf{factorization} of $ \sigma $.
Since $\pi$ is surjective, it is possible that the same
morphism is associated with two different automorphisms.
It is important to note right now, that with the point of view
that we adopt in this work,
two automorphisms that can be factored in the same way on $ G $,
will be indistinguishable.
Let $u$ and $v$ be two elements of the free group $ \mathbb F_m $.
We denote by $ [u, v] = uv u^ {-1} v ^ {-1} $ the \textbf {commutator} of $ u $ and $ v $.
The endomorphisms $ \sigma $ of free groups satisfy the relation:
$$
\sigma ([u, v]) = [\sigma (u), \sigma (v)].
$$
This relation gives hope to obtain interesting results
considering nilpotent  groups,
defined from relations
of commutators in the free group.
\vspace {6pt}
\\
We consider the situation where $G$ is the discret Heisenberg group
and will give some results for the group $ \mathbb F_2 $.
We begin by recalling some results
related to the Heisenberg group in Section \ref{intro:se:heis}
and we introduce transformations of
this space such as the \textbf{nilflows}
and the \textbf{niltranslations}.

\begin{proposition}
\label{nil:prop:autosubs}
The set of matrices with integer coefficients
form a lattice of the Heisenberg group
which is adapted for automorphisms of the free group $\mathbb F_2$.
\end{proposition}

\noindent
There is no object known at this time, which corresponds to the Rauzy fractal. But we believe such an object exists.
We obtain some results in this direction in Section \ref{nil:se:ex}, where we study a family of niltranslations connected with the "Fibonacci substitution".
The fact that these niltranslations come from substitutions, yields self-induced dynamical systems. 
The self-induction property corresponds to the self-similarity of the
Rauzy fractal under renormalization.
We show:

\begin{proposition}
\label{nil:th:rautoin}
Let $\phi$ be the golden mean.
The dynamical system given by the application:
$$
\begin{array}{cccc}
\left( \mathbb R / \mathbb Z \right)^2 & \longrightarrow & \left( \mathbb R / \mathbb Z \right)^2 \\
(y,z) & \mapsto & \left( y+\frac{1}{\phi^2}  \mbox{ , }  z+y -\frac{1}{2\phi^3}  \right)\\
\end{array}
\mbox{ is self-induced, minimal and uniquely ergodic. }
 $$
\end{proposition}

\noindent
The Heisenberg group has in its automorphisms group, some semi-simple and hyperbolic elements, that stabilize discrete Heisenberg group $\Gamma$, and preserve the center. 
Let G, be the group of unimodular automorphisms of the Heisenberg space. The space $G/$ stab$_G (\Gamma)$ is then a natural lattice space. 
The set of one-parameter flows is identified with the Lie algebra of the group, and we can then consider it, as a flat bundle over the moduli space. 
\\
The flow generated by a one-parameter group of semi-simple hyperbolic elements on $G$, induced on this bundle, a flow called \textbf{the renormalization flow}. 
L. Flaminio and G. Forni study  this flow in \cite{MR2218767}.
They deduce results on the distribution of flows in Heisenberg space by considering a co-homological equation. 
\\
We are interested by the periodic orbits of the renormalization flow and give an explicit calculation of the renormalized flow.
In Proposition \ref{nil:prop:autosubs}, we show that the periodic points of the renormalization flow
arise naturally from automorphisms acting on $\mathbb F_2$.
\\
We construct sections of flows adapted to these automorphisms.
We will see that  the first return of these flows into these sections, have remarkable properties.
The existence of such sections is not obvious, and we are currently unable to generalize this constructions to higher dimensions.
These applications  are conjugate to niltranslations. We obtain the following result:

\begin{theoreme}
\label{th7}
Let $ M = \begin{pmatrix} A & B \\ C & D \end{pmatrix}$
be a matrix with integer coefficients such that   $\mid det(M) \mid =1$.
\\
Assume that this matrix admits an eigenvalue  $\lambda$ with modulus $\mid \lambda \mid >1$.
We denote by $(\alpha,\beta)$ the eigenvector associated to $\lambda$ such that $\alpha+\beta=1$.
For every pair of integers $(n,m)$,
let:
\begin{equation}
\label{nil:eq:gamma}
\gamma =   \frac{ \alpha}{\lambda - \det(M)}  \left(  n - \frac{AC}{2} \right) +  \frac{ \beta }{\lambda - \det(M)}  \left( m - \frac{BD}{2} \right).
\end{equation}
Then, with the notations which we will introduce, the niltranslation by the element
$\begin{bmatrix} \alpha \\ \beta \\ \gamma + \frac{\alpha \beta}{2}  \end{bmatrix} $
on the nilmanifold
$$
\left\{   \begin{bmatrix} x \\ -x+n \\ z \end{bmatrix} ;  \begin{array}{ll} (x,z)\in \mathbb R^2 \\ n \in   \mathbb Z \end{array} \right\}
\mbox{ \Big / }
\left\{  \begin{bmatrix} n \\ m \\ p \end{bmatrix} ; (n,m,p)\in \mathbb Z^3 \right\}
\mbox{ is self-induced, minimal and uniquely ergodic. }
$$
\end{theoreme}

\noindent
In particular, this theorem states that 
each niltranslation which is the first return  map (with constant return time $1$)  of a nilflow periodic under renormalization, is self-induced.
Then, a naturel question arise : \textit{Do the self-induced niltranslation come from a periodic nilflow under renormalization ?}
We will see that the answer to this question is no and we start Section \ref{se:fin} by exhibiting a counterexample.
We also raise another difficulty.
We will see that it is possible that the niltranslations can be self-induced in areas that do not project well on the abelianisation.

\section{The Heisenberg group}
\label{intro:se:heis}

We recall here some properties of
the \textbf{Heisenberg group} $ \mathbb H _3 (\mathbb R) $, denoted $ \boldsymbol {X} $,
of real upper triangular 3$\times$3 matrices,
with ``1s'' on the diagonal.
The group law is given by:
$$
\boldsymbol{x } \bullet \boldsymbol{x' }=
\begin{bmatrix} x +x' \\ y +y' \\ z + z' + x y'  \end{bmatrix}
\mbox{ where }
\boldsymbol{x} = \begin{pmatrix} 1 & x& z \\ 0 & 1 & y \\ 0 & 0 & 1 \end{pmatrix}
=
\begin{bmatrix} x  \\ y \\ z  \end{bmatrix} .
\mbox{ We get }
\boldsymbol{x } ^{-1} = \begin{bmatrix} -x \\ -y  \\  xy - z  \end{bmatrix}.
$$
The identity element of the group is the identity matrix, denoted by
$\boldsymbol{1}$. 
The \textbf{commutator} of elements
$\boldsymbol{x }$ and $\boldsymbol{y }$ is:
$
[ \boldsymbol{x },\boldsymbol{y }] =  \boldsymbol{x } \bullet \boldsymbol{y }
\bullet \boldsymbol{x }^{-1} \bullet \boldsymbol{y }^{-1}.
$
The \textbf{center} of the group is:
$$
 \boldsymbol{Z}
 = \{ \boldsymbol{x } \in \boldsymbol{X}; 
 \mbox{ $ \forall \boldsymbol{y} \in \boldsymbol{X}$ , }
 [ \boldsymbol{x },\boldsymbol{y }] = \boldsymbol{1 }\}
  = \left\{  \begin{bmatrix} 0  \\ 0 \\ z  \end{bmatrix} 
   \in \boldsymbol{X}; z \in \mathbb R
   \right\}.
$$
We denote by  $\boldsymbol{p } :\boldsymbol{X} \rightarrow \mathbb R^2$,
the group homomorphism defined by: 
$\boldsymbol{p } ( \boldsymbol{x } ) = (x,y)$.
The following sequence is exact:
$$
 \xymatrix{
    1 \ar@{>}[r] &\boldsymbol{Z} \ar@{>}[r]^{i} &\boldsymbol{X} \ar@{>}[r]^{\boldsymbol{p}}   &\mathbb R^2 \ar@{>}[r]     &1
  }.
$$
We endow the space $\boldsymbol{X}$, with a metric $\boldsymbol{d}$,
which is invariant
by left multiplication.
(i.e. $\forall (\boldsymbol{x},\boldsymbol{y},\boldsymbol{z}) \in \boldsymbol{X}^3$: 
$\boldsymbol{d} ( \boldsymbol{x},\boldsymbol{y})=
\boldsymbol{d} (\boldsymbol{z} \bullet \boldsymbol{x},\boldsymbol{z} \bullet \boldsymbol{y})$).
It will be defined from the \textbf{group norm}:
$$
\left | \left |  \cdot \right | \right | _{\boldsymbol{X}}: \boldsymbol{X} \longrightarrow \mathbb R^+
\mbox{ defined by  }
\left | \left |  \begin{bmatrix} x  \\ y \\ z  \end{bmatrix}  \right | \right | _{\boldsymbol{X}}
= \left( \left( x^2+y^2 \right)^2 + \left( z - \frac{xy}{2} \right)^2 \right)^{\frac{1}{4}}.
$$
The application $\left | \left |  \cdot \right | \right | _{\boldsymbol{X}}$
is a group norm because it verifies the following three properties:
$\left | \left |  \boldsymbol{x}  \right | \right | _{\boldsymbol{X}} = 0$
if and only if $\boldsymbol{x}=\boldsymbol{1}$,
$\left | \left |  \boldsymbol{x}  \right | \right | _{\boldsymbol{X}} =\left | \left |  \boldsymbol{x}^{-1}  \right | \right | _{\boldsymbol{X}}$
and
$ \left | \left |  \boldsymbol{x} \bullet \boldsymbol{y} \right | \right | _{\boldsymbol{X}} \leq 
\left | \left |  \boldsymbol{x}  \right | \right | _{\boldsymbol{X}} +
\left | \left |   \boldsymbol{y}  \right | \right | _{\boldsymbol{X}}$
for all  $(\boldsymbol{x}, \boldsymbol{y}) \in \boldsymbol{X}^2$.
The metric $\boldsymbol{d}$ is defined for all $(\boldsymbol{x}, \boldsymbol{y}) \in \boldsymbol{X}^2$ by:
$
\boldsymbol{d} (\boldsymbol{x}, \boldsymbol{y})  = 
\left | \left |   \boldsymbol{x}^{-1} \bullet \boldsymbol{y}\right | \right | _{\boldsymbol{X}} .
$
\begin{figure}[H]
\centering
\includegraphics[width=4.5cm]{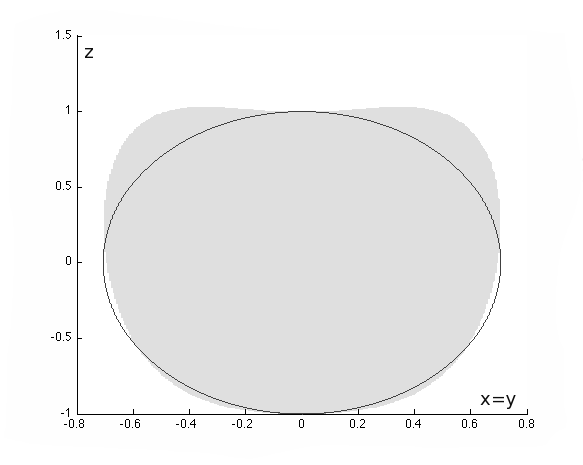}
\includegraphics[width=4.5cm]{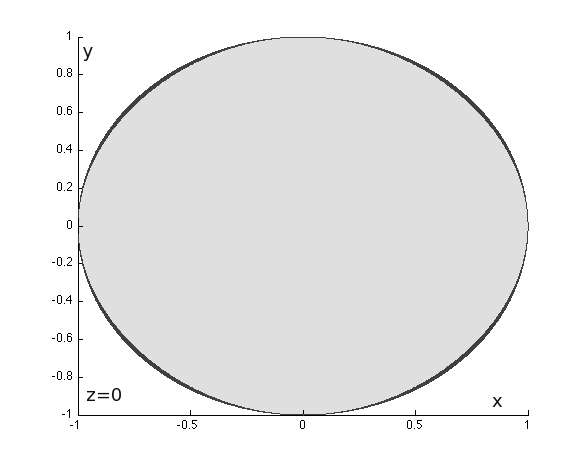}
\includegraphics[width=4.5cm]{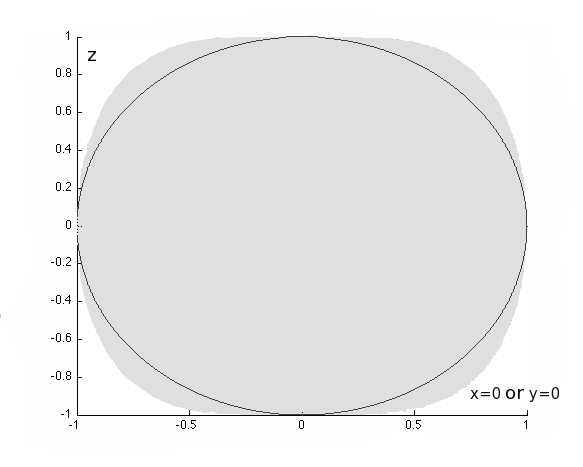}
\caption{We represent the unit ball of
$(\boldsymbol{X},\boldsymbol{d})$ and the unit sphere
of the standard Euclidean space
$\mathbb R^3$
in subspaces of $\boldsymbol{X}$
consisting of matrices $[x,y,z] $
satisfaying  respectively:
$x=y$, $z=0$, and $x=0$.}
\label{nil:fig:sphere}
\end{figure}
\noindent
Although the metric $\boldsymbol{d}$ 
and the standard Euclidean metric are different,
they induce the same topology on $\mathbb R^3$.
\vspace{6pt}
\\
Since
$\boldsymbol{Z} = [\boldsymbol{X},\boldsymbol{X}]$, the space $( \boldsymbol{X},\bullet,\boldsymbol{d})$
is a \textbf{nilpotent Lie group} of rank $2$.
It can be endowed with a differentiable structure. The tangent space to
$\boldsymbol{X}$ in $\boldsymbol{1}$, which is by definition its
\textbf{Lie Algebra}, is:
$$
\boldsymbol{ \mathfrak g}=
 \left\{\boldsymbol{ \mathfrak x} = 
\begin{pmatrix} 0 & \alpha & \gamma \\ 0 & 0 & \beta \\ 0 & 0 & 0 \end{pmatrix}
 ; (\alpha, \beta, \gamma) \in \mathbb R ^3
   \right\}.
   \mbox{ The elements of $\boldsymbol{ \mathfrak g}$ will be denoted }
   \boldsymbol{ \mathfrak x} = 
\begin{pmatrix} \alpha \\  \beta \\  \gamma  \end{pmatrix}
.
$$
Since the space $(\boldsymbol{X},\boldsymbol{d})$ is
connected, the exponential is a diffeomorphism from this space to
Lie algebra $\boldsymbol{ \mathfrak g}$.
$$
\exp :
\begin{array}{cccc}
 \boldsymbol{ \mathfrak g} &  \longrightarrow & \boldsymbol{ X} \\
\mathfrak{x} = \begin{pmatrix} \alpha \\  \beta \\  \gamma  \end{pmatrix} & \mapsto &
\boldsymbol{1} + \mathfrak{x} + \frac{1}{2}\mathfrak{x}^2
=
\begin{bmatrix} \alpha \\  \beta \\  \gamma + \frac{\alpha \beta}{2}  \end{bmatrix} 
\end{array}
\mbox{ and }
\log :
\begin{array}{cccc}
 \boldsymbol{ X}   &  \longrightarrow & \boldsymbol{ \mathfrak g}\\
\begin{bmatrix} \alpha \\  \beta \\  \gamma  \end{bmatrix} & \mapsto &
\begin{pmatrix} \alpha \\  \beta \\  \gamma  - \frac{\alpha \beta}{2}\end{pmatrix}
\end{array}.
$$
The \textbf{Lie bracket} in the Lie algebra is defined by:
 $
 \left [ \boldsymbol{ \mathfrak x} , \boldsymbol{ \mathfrak x} '\right]
=
\frac{1}{2}
\log \left( [ \exp \boldsymbol{ \mathfrak x} , \exp \boldsymbol{ \mathfrak x}' ]\right).$
\vspace{6pt}
\\
With these notations,
$
\exp ( \boldsymbol{ \mathfrak x} + \boldsymbol{ \mathfrak x'} )
\bullet \exp ( [\boldsymbol{ \mathfrak x},\boldsymbol{ \mathfrak x} '])
=
\exp  \boldsymbol{ \mathfrak x} \bullet \exp \boldsymbol{ \mathfrak x'} 
\mbox{ and }
\log (\boldsymbol{ x} \bullet  \boldsymbol{ x}')=
\log \boldsymbol{ x}  +  \log \boldsymbol{ x}'
+ \log ([ \boldsymbol{ x}  ,  \boldsymbol{ x}']) . 
$
\vspace{6pt}
\\
For any element
$\boldsymbol{ \mathfrak x} = \begin{pmatrix} \alpha \\  \beta \\  \gamma  \end{pmatrix} $
of the Lie algebra $\boldsymbol{ \mathfrak g}$, we denote by:
\begin{equation}
\label{nil:eq:ssgroupe}
\boldsymbol{G}_{\boldsymbol{ \mathfrak x}} = 
\left\{  \exp(t \cdot \boldsymbol{ \mathfrak x} ) \mbox{ ; } t \in \mathbb R 
\right\}
=
\left\{  \begin{bmatrix} \alpha t  \\ \beta t \\ \gamma t + \frac{\alpha \beta}{2} t^2 \end{bmatrix}
 \mbox{ ; } t \in \mathbb R 
\right\}
\mbox{  with the convention }
t \cdot \boldsymbol{ \mathfrak x}
=
\begin{pmatrix} \alpha t\\  \beta t \\  \gamma t \end{pmatrix} .
\end{equation}
These are the $\boldsymbol{1}$ \textbf{parameter sub-groups} of $\boldsymbol{X}$.
\vspace{6pt}
\\
Let $\mathbb H_3(\mathbb Z)=\boldsymbol{\Gamma}$ be the sub-group of $\boldsymbol{X}$
consisting of matrices with integer coefficients.
The following sequence is exact:
$$
 \xymatrix{
    1 \ar@{>}[r]
    &\boldsymbol{Z}\cap \boldsymbol{\Gamma} \ar@{>}[r]^{i} 
    & \boldsymbol{\Gamma} \ar@{>}[r]^{\boldsymbol{p}}
     &\mathbb Z^2 \ar@{>}[r]
     &1
  }.
$$
The metric $\boldsymbol{d}$
induces a metric   $\underline{ \boldsymbol{d}}$
on the quotient space $\boldsymbol{X} \setminus \boldsymbol{\Gamma}$
denoted $\underline{\boldsymbol{X}}$.
The space $\boldsymbol{X}$ acts isometrically by left translation on
$\underline{\boldsymbol{X}}$. There is
a unique probability measure invariant by this action, called the \textbf{Haar measure}.
By definition, $(\underline{\boldsymbol{X}}, \underline{ \boldsymbol{d}})$
is a \textbf{nilmanifold}. 
The space  $\underline{\boldsymbol{X}}$
is topologically isomorphic to the space $[0;1]^3$, with the following identifications:
$$
\begin{bmatrix} 0  \\ y \\ z  \end{bmatrix} \sim \begin{bmatrix} 1  \\ y \\ z  \end{bmatrix}
\mbox{, } \begin{bmatrix} x  \\ y \\ 0 \end{bmatrix} \sim \begin{bmatrix} x  \\ y \\ 1  \end{bmatrix}
\mbox{, } \begin{bmatrix} x  \\ 0 \\ z  \end{bmatrix} \sim \begin{bmatrix} x  \\ 1 \\ x+z \mbox{ mod }1  \end{bmatrix}
\mbox{ and } \begin{bmatrix} x  \\ 1 \\ z  \end{bmatrix} \sim \begin{bmatrix} x  \\ 0 \\ z-x \mbox{ mod }1  \end{bmatrix}.
$$
\begin{figure}[!h]
\centering
\includegraphics[width=4cm]{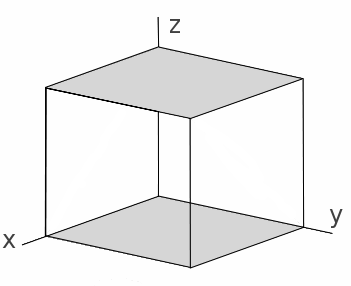}
\includegraphics[width=4cm]{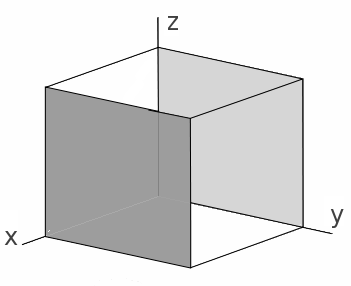}
\includegraphics[width=4cm]{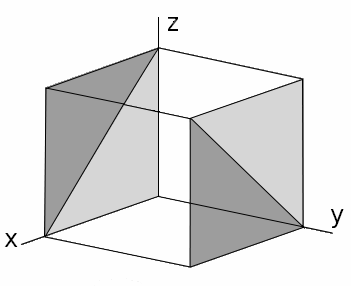}
\caption{Identification of the faces of the standard unit cube
to get the nilmanifold $\underline{\boldsymbol{X}}$.}
\label{nil:fig:identif}
\end{figure}
\\
The Haar measure of the space $\underline{\boldsymbol{X}}$, immersed in this fundamental domain, is the standard Lebesgue measure, denoted $\lambda^3$. 
There are three types of dynamical systems acting on spaces $\boldsymbol{X}$ and $\underline{\boldsymbol{X}}$ natural to consider and preserving the measure.
\vspace{6pt}
\\
The first family of applications is composed of continuous homomorphisms of  $\boldsymbol{X}$.
Since the work of G. Gelbrich in \cite{MR1326950},  we know that they are of the following form:
\begin{equation}
\label{nil:eq:hom}
\mathcal L :  \begin{bmatrix} x  \\ y \\ z  \end{bmatrix} \mapsto \begin{bmatrix} ax+by  \\ cx+dy \\  \frac{ac}{2} x^2+ (e- \frac{ac}{2})x + \frac{bd}{2} y^2+ (f- \frac{bd}{2})y + bc xy + (ad-bc)z \end{bmatrix}.
\end{equation}
$\mathcal L$ preserve the lattice $\boldsymbol{\Gamma}$ when
the coefficients $(a,b,c,d,e,f)$ are integers. In this case, the application $\mathcal L$ also acts on the quotient space $\underline{\boldsymbol{X}}$.
In addition, these applications preserve the Haar measure, if the coefficients satisfy the equation:  $\mid ad-bc \mid =1$.
\vspace{6pt}
\\
We are also interested in the action of  $1$ parameter subgroups on the space $\boldsymbol{X}$, given by equation (\ref{nil:eq:ssgroupe}):
\begin{equation} \label{nil:eq:flot}
\Phi ^t  _{\mathfrak x}:  \boldsymbol{x} \mapsto \boldsymbol{g} _{ \boldsymbol{\mathfrak x}} ^t  \bullet \boldsymbol{x} 
\mbox{, where } \boldsymbol{g} _{\boldsymbol{\mathfrak x}} ^t 
= \exp ( t \cdot \boldsymbol{\mathfrak x})= \begin{bmatrix} \alpha t  \\ \beta t \\ \gamma t + \frac{\alpha \beta}{2} t^2 \end{bmatrix}
\mbox{ with $t \in \mathbb R$ and $\boldsymbol{\mathfrak x} = \begin{pmatrix} \alpha \\  \beta \\  \gamma  \end{pmatrix} \in \boldsymbol{\mathfrak g}$},
\end{equation}
and their discrete time analogue, the action by left translations:
\begin{equation}
\label{nil:eq:trans}
T_{\boldsymbol{\mathfrak x}}  : 
\begin{array}{cccc} \boldsymbol{X} & \rightarrow & \boldsymbol{X} \\ \boldsymbol{x } & \mapsto &  \boldsymbol{g } _{\boldsymbol{\mathfrak x}}  \bullet \boldsymbol{x}  \end{array}
\mbox{ with $\boldsymbol{\mathfrak x} = \begin{pmatrix} \alpha \\  \beta \\  \gamma  \end{pmatrix} \in \boldsymbol{\mathfrak g}$ and }
\boldsymbol{g } _{\boldsymbol{\mathfrak x}}  = \exp({\boldsymbol{\mathfrak x}} ) = \begin{bmatrix} \alpha  \\ \beta \\ \gamma + \frac{\alpha \beta}{2} \end{bmatrix} .
\end{equation}
We will also denote these maps $\Phi  _{\alpha,\beta,\gamma}$ and $T_{\alpha,\beta,\gamma} $.
These applications act naturally on the quotient space  $\underline{\boldsymbol{X}}$, denoted them by $\underline{T}$ and $\underline{\Phi }$.
These classes of systems, called \textbf{niltranslations} and \textbf{nilflows}, have been widely studied. Let us quote here two central results.
For $\underline{\boldsymbol{x}} \in \underline{\boldsymbol{X}}$, we put $\mathcal O (\underline{\boldsymbol{x}} ) =  \overline{\{ \underline{T} ^n ( \underline{\boldsymbol{x}} );n \in \mathbb N\}}$.

\begin{theor}[ E. Lesigne \cite{MR1116647}] \label{th:emmanuelnil}
The system $(\mathcal O (\underline{\boldsymbol{x}} ), \underline{T})$ is minimal and uniquely ergodic.
\end{theor}

\begin{theor}[L. Auslander, L. Green and F. Hahn  \cite{MR0126504},\cite{MR0167569}] \label{nil:th:AGH}
The flow $\left(\underline{\Phi }^t  _{\alpha,\beta,\gamma}\right)_t$ on $\underline{\boldsymbol{X}}$ is minimal if and only if it is uniquely ergodic if and only if the coefficients  $\alpha$ and $\beta$  are linearly independent.
\end{theor}

\noindent
We consider the flow $\Psi ^t$ defined by:
$ \boldsymbol{c} ^t  = \begin{bmatrix} 0 \\ 0 \\ t  \end{bmatrix} \mbox{ and } \Psi ^t (\boldsymbol{x} )  =  \boldsymbol{c} ^t  \bullet \boldsymbol{x} \mbox{ for $t\in \mathbb R$}. $
 \begin{lemme}  \label{nil:lemme:flotcomm}
 $\Phi _{{\alpha',\beta',\gamma'} } ^s \circ \Phi _{{\alpha,\beta,\gamma} } ^t =   \Phi _{{\alpha,\beta,\gamma} } ^t \circ \Phi _{{\alpha',\beta',\gamma'} } ^s \circ \Psi ^{\Delta(t-s)}$
 where $\Delta =  \beta  \alpha ' -  \beta'  \alpha $.  In particular,   $\Psi  ^s \circ \Phi _{{\alpha,\beta,\gamma} } ^t =  \Phi _{{\alpha,\beta,\gamma} } ^t \circ \Psi ^s$. 
 \end{lemme}
 
\begin{proof}
Just calculate the following expressions:
$$
\begin{array}{cccccc}
& &
\Phi _{\alpha',\beta',\gamma'} ^s \circ \Phi _{{\alpha,\beta,\gamma} } ^t  \begin{bmatrix}  x \\ y  \\  z  \end{bmatrix}
&  = &
  \begin{bmatrix}  x +   t \alpha + s \alpha '\\ y +   t \beta + s \beta '\\ 
 z  + y t \alpha + (y + t \beta) s \alpha ' +  \gamma t + \gamma ' s + \frac{\alpha \beta }{2} t^2  +  \frac{\alpha' \beta' }{2} s^2
 \end{bmatrix} \\
&\mbox{ and }&
\Phi _{{\alpha,\beta,\gamma}  } ^t \circ \Phi _{\alpha',\beta',\gamma'} ^s \begin{bmatrix}  x \\ y  \\  z  \end{bmatrix}
&  = &
  \begin{bmatrix} 
x +   s \alpha ' +  t \alpha   \\
y +   s \beta ' + t \beta \\ 
 z  + y s \alpha' + (y + s \beta') t \alpha  +
 \gamma' s + \gamma  t + \frac{\alpha' \beta' }{2} s^2
 +  \frac{\alpha \beta }{2} t^2
 \end{bmatrix} .
 \end{array}$$
\end{proof}
\noindent
The group norm verifies some properties with respect to the introduced objects.
For any element $\boldsymbol{x} \in \boldsymbol{X}$,
the flow $\Phi_{\log \boldsymbol{x}} $
is the unique flow satisfying
$\Phi^1(\boldsymbol{0}) = \boldsymbol{x}$. 
If $\boldsymbol{x}=[x,y,z]$,
the group norm verifies:
$$
\left | \left |  \Phi ^t _{\log \boldsymbol{x}} ( \boldsymbol{1})  \right | \right | _{\boldsymbol{X}}
=
 \left( t^4 \left( x^2+y^2 \right)^2 +t^2 \left( z - \frac{xy}{2} \right)^2 \right)^{\frac{1}{4}}.
$$
For every real $t$, we also consider the expansion of space
$ \mathcal D ^t :  \boldsymbol{X}  \longrightarrow \boldsymbol{X}$,
such that
$ \mathcal D ^t ([x,y,z])= [xt , y t , z t^2 ] $.
The group norm verifies:
$
\left | \left |  \mathcal D ^t \boldsymbol{x} \right | \right | _{\boldsymbol{X}}
 = \mid t \mid  \cdot \left | \left | \boldsymbol{x} \right | \right | _{\boldsymbol{X}}
\mbox{ for }
\boldsymbol{x}\in \boldsymbol{X}.
$
\vspace{6pt}
\\
A significant difference with the abelian situation, is that for every real $t\notin \{ 0,1\}$,
the application of $ \boldsymbol{X}$ into itself defined by:
$
 \boldsymbol{x}  \mapsto   \Phi ^t _{\log \boldsymbol{x}} ( \boldsymbol{1}),  
$
is not a group homomorphism of $\boldsymbol{X}$.
For more details on the left invariant metric of this group,
we refer to  \cite{MR2035027}, \cite{MR1106594}, \cite{LEE} and \cite{MR2478814}.

\section{Symbolic approach}
\label{nil:se:sym}

We start by proving Proposition \ref{nil:prop:autosubs},
which makes the link between the automorphisms of
the free group on two generators and the morphisms of the lattice
$\mathbb H_3(\mathbb Z) = \boldsymbol{\Gamma}$.
The generators of this lattice will be noted: 
$$
\boldsymbol{n}_a =   \begin{bmatrix}  1 \\ 0  \\  0  \end{bmatrix}
\mbox{, }
\boldsymbol{n}_b =   \begin{bmatrix}  0 \\ 1  \\  0  \end{bmatrix}.
\mbox{  Also let  }
\boldsymbol{n}_{a^{-1}} = \boldsymbol{n}_a^{-1}
\mbox{, }
\boldsymbol{n}_{b^{-1}} = \boldsymbol{n}_b^{-1}
\mbox{ and }
\boldsymbol{n} = [ \boldsymbol{n}_a,\boldsymbol{n}_b ] = 
  \begin{bmatrix}  0 \\ 0 \\  1  \end{bmatrix} \in \boldsymbol{Z}\cap \boldsymbol{\Gamma}.
$$
Let $ \sigma $ be a automorphism of the free group $ \mathbb F_2 $. It can be written 
\begin{equation}
\label{nil:eq:substitution}
\sigma :
\left\{
\begin{array}{cccc}
a & \longrightarrow &  \xi_1 \dots \xi_{l_a} \\
b  & \longrightarrow &  \zeta_1\dots \zeta_{l_b}
\end{array}
\right.
\mbox{ with each $(\xi_i)_{1\leq i\leq l_a}$ and $(\zeta_i)_{1\leq i\leq l_b}$  in $\{a,b,a^{-1},b^{-1} \}$.  }
\end{equation}
\begin{equation}
\label{eq:ra}
\mbox{We associate to $\sigma$ the endomorphism $\mathfrak{S}_{\sigma}$ of $\Gamma$ defined by }
 \mathfrak S_\sigma (\boldsymbol{n} _a) =  \boldsymbol{n}_{\xi_1} \dots   \boldsymbol{n}_{\xi_{l_a}}
 \mbox{ and }
 \mathfrak S_\sigma (\boldsymbol{n} _b) = \boldsymbol{n}_{\zeta_1} \dots   \boldsymbol{n}_{\zeta_{l_b}}. 
 \end{equation}
This object is well defined since
$\mathfrak S_\sigma (\boldsymbol{n}) =  \mathfrak S_\sigma (\boldsymbol{n} _a) 
 \cdot  \mathfrak S_\sigma (\boldsymbol{n} _b) \cdot  \mathfrak S_\sigma (\boldsymbol{n} _a) ^{-1} \cdot \mathfrak S_\sigma (\boldsymbol{n} _b)^{-1}$.
\begin{proof}[Proof of Proposition \ref{nil:prop:autosubs}]
The application $\mathfrak S : $Aut$(\mathbb F_2)\to$ End$(\boldsymbol{\Gamma})$ is a morphism. 
We will show that $\mathfrak S ($Aut$(\mathbb F_2))=$Aut$(\boldsymbol{\Gamma})$.
\textit{We start by showing that for every $\sigma\in $ Aut$(\mathbb F_2)$, then $\mathfrak S_\sigma \in$ Aut$(\boldsymbol{\Gamma})$.}
We note $M$ the action of $\boldsymbol{p} \circ \mathfrak S_\sigma$
on $\mathbb R^2$. From Equation (\ref{nil:eq:hom}), we know that $\mathfrak S_\sigma (\boldsymbol{n})  = det(M) \boldsymbol{n}$.
Since $det(M) \in \{-1,1\}$, $\mathfrak S_\sigma $ is an automorphism of $\boldsymbol{\Gamma}$.
For more details, we refer to \cite{MR1970385} and \cite{MR924156}.
\textit{It only remains to verify that the map $ \mathfrak {S} $ is surjective.}
\vspace{6pt}
\\
Consider an endomorphism $ \mathcal L $ given 
by equation (\ref{nil:eq:hom}) of Section \ref{intro:se:heis}. We have 
$
\mathcal L(\boldsymbol{n} _a)   =  \begin{bmatrix}  x_1 \\ x_2  \\  x_3  \end{bmatrix}
\mbox{ and }
\mathcal L(\boldsymbol{n} _b)   =  \begin{bmatrix}  y_1 \\ y_2  \\  y_3  \end{bmatrix} .
$
We define the automorphisms $\sigma_1$, $\sigma_2$, $\sigma_3$ and $\sigma_4$, defined by:
$$
\sigma_1 : \left\{ \begin{array}{cccc} a & \longrightarrow & ab\\ b  & \longrightarrow &  b \end{array} \right. ,
\sigma_2 : \left\{ \begin{array}{cccc} a & \longrightarrow & ab\\ b  & \longrightarrow &  a \end{array} \right. ,
\sigma_3 : \left\{ \begin{array}{cccc} a & \longrightarrow & a\\ b  & \longrightarrow &  ba \end{array} \right. \mbox{ and }
\sigma_4 : \left\{ \begin{array}{cccc} a & \longrightarrow & b\\ b  & \longrightarrow &  ab \end{array} \right.
.
$$
We  write ${\mathfrak S}_i=\mathfrak S_{\sigma_i}$, for $i\in \{1,2,3,4\}$.
It is conventional to verify that there exits an integer $k$, $(n,m)\in \mathbb Z^2$, $(u_j)\in\{1,4\}^k$, and $(\epsilon_j)\in \{-1,1\}^k$ such that
$$
 \prod \limits_{j=1}^k  {\mathfrak S} _{u_j} ^{\epsilon_j} \circ \mathfrak S (\boldsymbol{n}_a) =  \boldsymbol{n}_a \bullet 
 \begin{bmatrix}  0 \\ 0 \\  n  \end{bmatrix} 
 \mbox{ and }
 \prod \limits_{j=1}^k  {\mathfrak S} _{u_j} ^{\epsilon_j} \circ \mathfrak S (\boldsymbol{n}_b) = \boldsymbol{n}_b \bullet 
 \begin{bmatrix}  0 \\ 0 \\  m  \end{bmatrix} .
$$
Then define the following automorphisms:
$$
\sigma_5 :
\left\{
\begin{array}{cccc}
a & \longrightarrow & b^{-1} a b\\
b  & \longrightarrow &  b
\end{array}
\right.,
\sigma_6 :
\left\{
\begin{array}{cccc}
a & \longrightarrow & a \\
b  & \longrightarrow &  a^{-1} b a
\end{array}
\right.
\mbox{ and 
 ${\mathfrak S}_i=\mathfrak S_{\sigma_i}$, for $i\in \{5,6\}$. }
 $$
By a calculation, we can verify that
$$
( {\mathfrak S} _ 5 )^n
 \circ
 ( {\mathfrak S} _ 6 )^{-m}
  (\boldsymbol{n}_a) =  \boldsymbol{n}_a \bullet 
 \begin{bmatrix}  0 \\ 0 \\  n  \end{bmatrix} 
 \mbox{ and }
( {\mathfrak S} _ 6 )^n
 \circ
 ( {\mathfrak S} _ 5 )^{-m} (\boldsymbol{n}_b) = \boldsymbol{n}_b \bullet 
 \begin{bmatrix}  0 \\ 0 \\  m  \end{bmatrix}.
$$
So, with $\sigma =  \prod \limits_{j=1}^k  {\sigma} _{u_{k-j+1}} ^{-\epsilon_{k-j+1}} \circ ( {\mathfrak S} _ 6 )^n
 \circ ( {\mathfrak S} _ 5 )^{-m}$, we have $\mathcal L =\mathfrak S_\sigma$.
\end{proof}

\noindent
Throughout this work, we deal with the general case.
However, we will treat the \textbf{Fibonacci substitution}  to illustrate our results:
$$
\tau : \left\{
\begin{array}{cccc}
a & \rightarrow & a b,\\
b & \rightarrow & a .\\
\end{array} \right. 
$$
We denote by $u=(u_k)_{k\geq 1}= abaaaba \dots \in \{a,b\}^{\mathbb N}$
the infinite word, fixed point of this substitution, and
$\phi = \frac{1+\sqrt{5}}{2}$
the golden mean.
We begin with the Fibonacci substitution. 
We define a sequence $(\boldsymbol{x_k})_{k\geq 0} \in \boldsymbol{X}^{\mathbb N}$ as follows: 
$$
\boldsymbol{x_0} \ = \   \boldsymbol{1}
\mbox{ and for $k\geq 1$: }  \boldsymbol{x_{k+1}} \ = \ \boldsymbol{x_k} \bullet \boldsymbol{n_{u_k}} .
$$
We call this sequence, \textbf{the broken line} associated with the substitution $\tau$ in $\boldsymbol{X}$.
For any integer $k$, we write:
$$
\boldsymbol{x_k}  = 
 \boldsymbol{1}  \bullet \prod \limits_{i=1} ^k \boldsymbol{n_{u_i}} = 
\begin{bmatrix} a_k  \\ b_k \\ c_k  \end{bmatrix}.
$$
A direct calculation shows that $a_0=b_0=c_0=0$, and for any integer $k\geq 1$:
$$
a_k = \# \left\{ 1\leq i \leq k \mbox{ ; } u_i = a \right\}
\mbox{ , }
b_k = \# \left\{ 1\leq i \leq k \mbox{ ; } u_i = b \right\}
\mbox{ and } c_k  = \# \left\{ 1\leq i <j \leq k \mbox{ ; } u_i = a \mbox{ and } u_j= b  \right\}.
$$
For any integer $k$, the quantity $c_k$ can be viewed ``geometrically'' 
by the area of the gray zone in Figure \ref{nil:fig:ck}.
\begin{figure}[H]
\centering
\includegraphics[width=9cm]{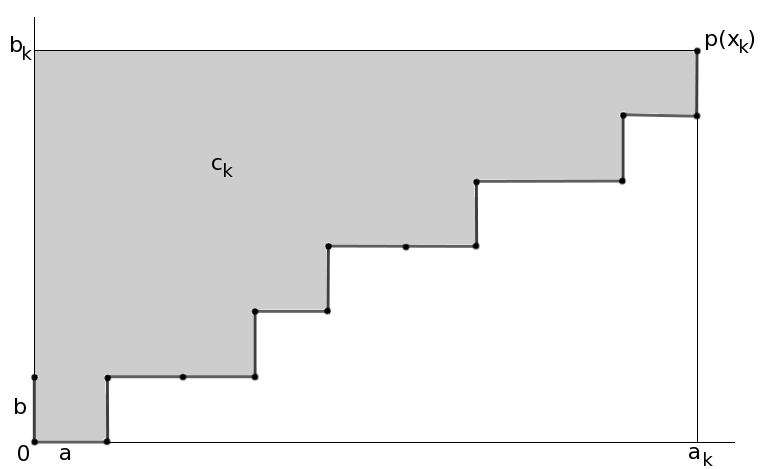}
\caption{Projections by $\boldsymbol{p} $ of the broken line $(\boldsymbol{x}_k)_k$ in $\mathbb R^2$.}
\label{nil:fig:ck}
\end{figure}
\noindent
The challenge is to find an element 
$ \boldsymbol{g} =
\begin{bmatrix} \alpha  \\ \beta \\ \gamma  \end{bmatrix}
\in \boldsymbol{X}$
such that the sequence
$(\boldsymbol{g} ^k  \bullet \boldsymbol{x_k})_k$ is bounded.
\vspace{6pt}
\\
In particular, in order to bound the sequence of elements
$\Big{(} \boldsymbol{p}(\boldsymbol{g} ^k (\alpha,\beta,\gamma) \bullet \boldsymbol{x_k} \Big{)}_k$
of $\mathbb R^2$, the element $\boldsymbol{g}$ should be choosen as: 
$$
\boldsymbol{g}_\theta = 
\begin{pmatrix} 1 & \frac{-1}{\phi} & \theta\\ 0 & 1 & \frac{-1}{\phi^2} \\ 0 & 0 & 1 \end{pmatrix}
=
\begin{bmatrix}  \frac{-1}{\phi} \\ \frac{-1}{\phi^2} \\  \theta \end{bmatrix}
.
$$
For this reason, we focus in Section \ref{nil:se:ex}  on the left action of matrices
$\boldsymbol{g}_\theta$ on the quotient space $\underline{\boldsymbol{X}}$.
\vspace{6pt}
\\
For any automorphism $\sigma$ of $\mathbb F_2$, we will use the following notation:
\begin{equation}
\label{nil:eq:auto}
M_\sigma = [ m^{\epsilon,\epsilon'} ]_{(\epsilon,\epsilon') \in \{a,b \}^2}
= \boldsymbol{p}\circ \mathfrak S _\sigma \mbox{ and }
\mathfrak S _\sigma: 
\begin{array}{cccc}
\boldsymbol{X}  & \rightarrow & \boldsymbol{X}  \\
\begin{bmatrix} x  \\ y \\ z  \end{bmatrix} & \mapsto & 
\begin{bmatrix} m^{a,a} x + m^{a,b} y  \\ m^{b,a} x + m^{b,b} y  \\ 
\det (M_\sigma) z + P_\sigma(x,y)  \end{bmatrix}\\
\end{array}  ,
\end{equation}
$$
\mbox{ where }
P_\sigma (x,y) = \frac{m^{a,a}m^{b,a}}{2} x(x-1) + \frac{m^{a,b}m^{b,b}}{2} y(y-1)+ m^{a,b} m^{b,a} xy 
+ n^{a,b}_a x + n^{a,b}_b y.
$$
We  notice immediately that the map $ \mathfrak S _ \sigma $ 
is invertible if and only if the matrix 
$ M_\sigma $ is itself invertible.\textit{We will always assume this to hold.} 
For the Fibonacci substitution, this automorphism is:
\begin{equation}
\label{nil:eq:fauto}
\mathfrak S _\tau : 
\begin{array}{cccc}
\boldsymbol{X}  & \rightarrow & \boldsymbol{X} \\
\begin{bmatrix} x  \\ y \\ z  \end{bmatrix} & \mapsto & 
\begin{bmatrix} x+y  \\ x \\ -z+x(x+1)/2+xy  \end{bmatrix}\\
\end{array}  .
\end{equation}
In the proof of the following proposition we will see that under some assumptions on the matrix $ M_\sigma $,
we can associate to these automorphisms, some characteristic flows.

\begin{proposition}
\label{nil:prop:flotcomm}
Let $\lambda$ be a real eigenvalue of the matrix 
$M_\sigma$ which is not equal to
the determinant of the matrix.
$M_\sigma$. Let $(\alpha,\beta)$ be an eigenvector of the matrix 
associated to the eigenvalue $\lambda$.
Then, there exists a unique real $\gamma$, such that
the flow $\Phi  _{\alpha,\beta,\gamma}^t$ satisfies: 
$$
\mathfrak S \circ \Phi _{\alpha,\beta,\gamma} ^t   \circ \mathfrak S ^{-1} =  \Phi _{\alpha,\beta,\gamma}^{\lambda t} .
$$
The value of $\gamma$ is
$
\gamma =   \frac{ \alpha}{\lambda - \det(M_\sigma)} 
\left(  n_a ^{a,b} - \frac{m^{a,a}m^{b,a}}{2} \right)
+  \frac{ \beta }{\lambda - \det(M_\sigma)} 
\left(  n_b ^{a,b} - \frac{m^{b,a}m^{b,b}}{2} \right).
$
\end{proposition}
\begin{proof}
We denote the flow $\Phi_{\alpha,\beta,\gamma}$ defined in (\ref{nil:eq:flot}) by $\Phi$.
A direct calculation gives: 
$$
\mathfrak S \circ \Phi  ^t \begin{bmatrix} x  \\ y \\ z  \end{bmatrix} 
=
\begin{bmatrix} 
m^{a,a} x + m^{a,b} y + t ( m^{a,a} \alpha + m^{a,b} \beta ) \\
m^{b,a} x + m^{b,b} y + t ( m^{b,a} \alpha + m^{b,b} \beta )  \\ 
\left[ z + y t \alpha +  \gamma t + \frac{\alpha \beta }{2}  t ^2 \right] \det(M_\sigma)
+ P_\sigma (x+t\alpha,y+t \beta)
 \end{bmatrix} 
$$
$$
\mbox{ and }
\Phi  ^{\lambda t} \circ \mathfrak S  \begin{bmatrix} x  \\ y \\ z  \end{bmatrix} 
=
\begin{bmatrix} 
m^{a,a} x + m^{a,b} y + \lambda t \alpha  .\\
m^{b,a} x + m^{b,b} y + \lambda  t \beta   .\\ 
 z \det(M_\sigma) + \lambda t \alpha ( m^{b,a}x + m^{b,b} y ) +
 \gamma \lambda t  + \frac{\alpha \beta }{2}  ( \lambda t )^2
 + P_\sigma (x,y)
 \end{bmatrix} .
$$
It is therefore necessary to solve the system: 
$$
\mbox{ (S$_1$) : }
\left\{
\begin{array}{llll}
& &\begin{array}{cccc}
m^{a,a} x + m^{a,b} y + t ( m^{a,a} \alpha + m^{a,b} \beta )   
& = &
m^{a,a} x + m^{a,b} y + \lambda t \alpha ,
\\
m^{b,a} x + m^{b,b} y + t ( m^{b,a} \alpha + m^{b,b} \beta )  
& =  &
m^{b,a} x + m^{b,b} y + \lambda  t \beta ,
\\
\end{array} \\
& &
\left[ z + y t \alpha +  \gamma t + \frac{\alpha \beta }{2}  t ^2 \right] \det(M_\sigma)
+ P_\sigma (x+t\alpha,y+t \beta) =
\\
& &\mbox{   }\mbox{   }\mbox{   }\mbox{   }\mbox{   }
 z \det(M_\sigma) + \lambda t \alpha ( m^{b,a}x + m^{b,b} y ) +
 \gamma \lambda t  + \frac{\alpha \beta }{2}  ( \lambda t )^2
 + P_\sigma (x,y).
 \end{array}
\right.
$$
Since the vector $ (\alpha, \beta) $ is an eigenvector associated to the eigenvalue $\lambda $,
the first two equations are verified. It remains to consider the third equation. 
It is solved as follows: 
$$
\begin{array}{llll}
P_\sigma(x+t\alpha , y + t \beta ) &  = & 
P_\sigma(x,y) + \frac{m^{a,a} m^{b,a}}{2} (2x-1+t \alpha) t \alpha + \frac{m^{a,b} m^{b,b}}{2} (2y-1+t \beta)t \beta \\
& & +  m^{a,b} m^{b,a} t ( \alpha y + \beta x + t \alpha \beta) + n_a ^{a,b} t \alpha +  n_b ^{a,b} t \beta.
 \end{array}
$$
We must respectively cancel the terms in $t^2$, $x$, $y$ and $t$
in the third line of system (S$_1$). Thus, we must solve the system: 
$$
\mbox{ (S$_2$) : }
\left\{
\begin{array}{cccc}
\frac{\alpha \beta}{2}  \lambda ^2  & = &\frac{\alpha \beta}{2} \det(M_\sigma) + \alpha^2  \frac{m^{a,a} m^{b,a}}{2} 
+ \beta ^2  \frac{m^{a,b} m^{b,b}}{2} + \alpha \beta m^{a,b} m^{b,a},\\
\lambda \alpha m^{b,a} & = &   m^{a,a} m^{b,a} \alpha +m^{a,b} m^{b,a}  \beta ,\\
\lambda \alpha m^{b,b} & = &  \alpha \det(M_\sigma) + m^{a,b} m^{b,b}  \beta +  m^{a,b} m^{b,a}   \alpha ,\\
\gamma \lambda & = & 
\gamma \det(M_\sigma)  -
 \alpha  \frac{m^{a,a} m^{b,a}}{2}  -  \beta  \frac{m^{a,b} m^{b,b}}{2} 
+ n_a ^{a,b} \alpha +  n_b ^{a,b}  \beta.
\end{array} 
\right. 
$$
Since $ (\alpha, \beta) $ is an eigenvector of the matrix, 
the first three equations of system (S$_2 $) are always satisfied. Indeed, we observe: 
$$
\left\{
\begin{array}{cccc}
(m^{a,a}\alpha + m^{a,b} \beta)(m^{b,a}\alpha + m^{b,b} \beta) & = &
(\lambda \alpha) \cdot (\lambda \beta) =  \alpha \beta \lambda^2 , \\
\mbox{ et } m^{a,a}m^{b,b} - m^{a,b}m^{b,a} & = & \det(M_\sigma).
\end{array}
\right.
$$
The last line of system (S$_2$) has a solution if $\lambda \neq \det(M_\sigma)$  and we find (\ref{nil:eq:gamma}):
$$
\gamma = \frac{ \alpha}{\lambda - \det(M_\sigma)} 
\left(  n_a ^{a,b} - \frac{m^{a,a}m^{b,a}}{2} \right)
+  \frac{\beta}{\lambda - \det(M_\sigma)} 
\left(  n_b ^{a,b} - \frac{m^{b,a}m^{b,b}}{2} \right).
$$
 \end{proof}
 \noindent
For example, we can define these flows for the Fibonnaci substitution: 
\begin{equation}
\label{nil:eq:fflot}
 \Phi _{\phi} ^t   \begin{bmatrix} x  \\ y \\ z  \end{bmatrix}  =
 \begin{bmatrix} 
x +   t \frac{1}{\phi}  \\
y +   t \frac{1}{\phi^2}  \\ 
 z  + \frac{t(t+1)}{2\phi^3} + \frac{1}{\phi}yt
 \end{bmatrix} 
 \mbox{ and }
 \Phi _{-1/\phi} ^t  \begin{bmatrix} x  \\ y \\ z  \end{bmatrix}  =
 \begin{bmatrix} 
x +   t \frac{1}{\phi^2}  \\
y -   t \frac{1}{\phi}  \\ 
 z  + \frac{1}{\phi^2}  t - \frac{t(t-1)}{2\phi^3} + \frac{1}{\phi^2}yt
 \end{bmatrix} .
\end{equation}

\section{Example of a special niltranslation}
\label{nil:se:ex}

In this section,  we consider the left action of the matrix: 
$\underline{\boldsymbol{g}}_\theta = \begin{pmatrix} 1 & \frac{-1}{\phi} & \theta\\ 0 & 1 & \frac{-1}{\phi^2} \\ 0 & 0 & 1 \end{pmatrix}$
on the group $\underline{\boldsymbol{X}}$.
\vspace{6pt}
\\
We choose a fundamental domain of the quotient space $\underline{\boldsymbol{X}}$ depending on a parameter $s$:
$$
\boldsymbol{X}^s = 
\left\{ \boldsymbol{x}  = \begin{bmatrix} x \\ y \\ z \end{bmatrix} \mbox{ such that }  s \leq x  \leq 1+s \mbox{, } -s-1 \leq y \leq -s \mbox{ and } z\in [0,1] \right\}.
$$
Recall that, since the matrix $ \boldsymbol{n} $ introduced  in the previous section, is in the center of the group, we are free 
to quotient by the extremal coordinate ``$z$'' modulo $1$ at any time.  In particular, we can consider $ \theta $ modulo $ 1 $. 
\vspace{6pt}
\\
We restrict ourselves to study the action of  $\underline{\boldsymbol{g}}_\theta$ on:
$\boldsymbol{Y ^s} = \left\{ \boldsymbol{x}  = \begin{bmatrix} -y \\ y \\ z \end{bmatrix} \mbox{ such that } -s-1 \leq y \leq -s \mbox{ and } z\in [0,1] \right\}$.
\vspace{6pt}
\\
Since $ - \frac{1}{\phi}-  \frac{1}{\phi^2} \in \mathbb Z$, the element $\underline{\boldsymbol{g}}_\theta$ acts by translation on $\underline{\boldsymbol{Y}} ^s $.
Our goal will be to induce this application on: 
$$
\boldsymbol{Y ^s_{\mbox{\footnotesize Ind}}} = \left\{ \boldsymbol{x}  = \begin{bmatrix} -y \\ y \\ z \end{bmatrix} \mbox{ such that } -s-1   \leq y \leq -s - 1 + \frac{1}{\phi^2}   \mbox{ and } z\in [0,1] \right\}.
$$ 
\begin{figure}[H]
\centering
\includegraphics[width=9cm]{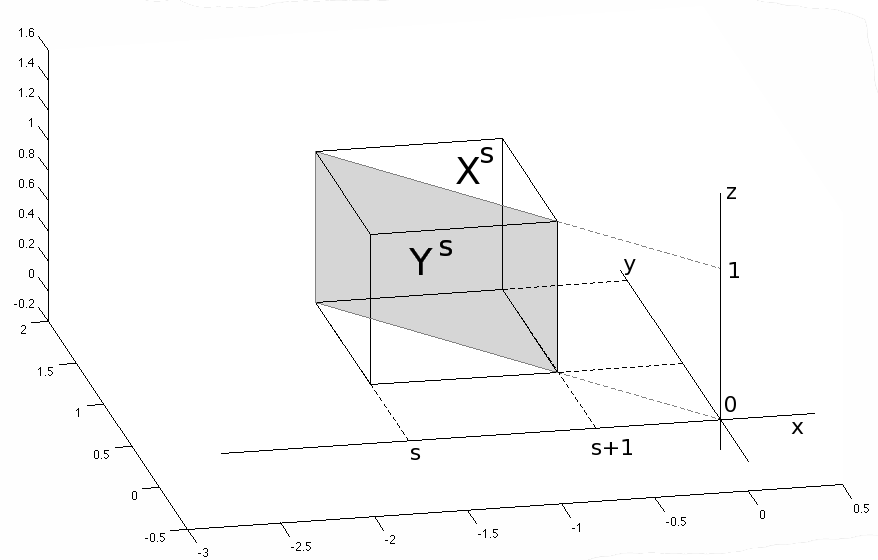}
\caption{Representation of $\boldsymbol{X}_s$ and $\boldsymbol{Y}_s$. }
\label{nil:fig:zoneniltranslationquotient}
\end{figure}
\noindent
We will prove the following result at the end of this section: 
\begin{proposition}
\label{nil:th:nil}
For any parameters $(s,s')\in \mathbb R^2$ and any angle $\theta \in \mathbb R$, there exists a map, called renormalization,
$\widetilde{\Phi} : \boldsymbol{Y ^s_{\mbox{\footnotesize  Ind}}}  \mapsto  \boldsymbol{Y ^{s'}}  $
and an angle $\theta'$, such that the first return application of $\underline{\boldsymbol{g}}_\theta$
on $\boldsymbol{Y ^s _{\mbox{ \footnotesize  Ind}}} $ is conjugated via $\widetilde{\Phi}$ to the action of $ \underline{\boldsymbol{g}}_\theta$ on
$ \boldsymbol{Y ^{s'}}  $. The angle $\theta'$ is given by:
$ \theta'=\phi^2 \theta + \phi^2 (s+1) - (s'+1)$.
\end{proposition}

\noindent
In particular, for the parameters $s=s'=-1$ and $\theta=0$, the action $ \underline{\boldsymbol{g}}_\theta$ 
on $ \boldsymbol{Y ^{-1}}$ is conjugated to the application $T: \left( \mathbb R / \mathbb Z \right)^2 \longrightarrow \left( \mathbb R / \mathbb Z \right)^2 $, defined by:
$$
T(y,z) = (y-\frac{1}{\phi^2} \mbox{ mod }1, z + \psi (y)\mbox{ mod }1)
\mbox{ where } \psi(y) = -\phi y+ \frac{-1}{\phi} \mbox{ if $0\leq y \leq \frac{1}{\phi^2}$ and } \psi(y) = -\frac{y}{\phi} \mbox{ otherwise.}
$$
The above calculations assure us that this application is self-induced.  A direct calculation shows that: 
$$
\psi(0)=\psi(1)= \frac{-1}{\phi} \mbox{ and } \psi^+ \left( \frac{1}{\phi^2} \right)- \psi ^- \left( \frac{1}{\phi^2} \right) = -\frac{1}{\phi^3}- \left( -\frac{2}{\phi^2} \right) =-1. 
$$
The application $\psi$  defines a continuous and Lipschitz map in the torus into itself  of degree $1$.
Thus after the work of H. Furstenberg \cite{MR0133429}, the system is uniquely ergodic. We also shown that it is self-induced.
\vspace{6pt}
\\
We put $p(y) =- \frac{1}{2}y^2 -  \frac{1}{2} y$, then $\psi(y) = p(y-\frac{1}{\phi^2} \mbox{ mod }1) - p(y) -y +\frac{1}{2\phi^3}$ for all $y\in[0,1]$.
Thus the map $ T $ is conjugate with the application of the torus   $\left( \mathbb R / \mathbb Z \right)^2 $ into itself defined by: 
$$
(y,z) \mapsto \left( y-\frac{1}{\phi^2}  \mbox{ mod }1, z-y +\frac{1}{2\phi^3} \mbox{ mod }1 \right),
$$
which is self-induced. \textit{We have therefore proved the following result:}
\vspace{6pt}
\\
\textbf{Proposition \ref{nil:th:rautoin}.}\textit{
Let $\phi$ be the golden mean.
The dynamical system given by the application defined from $\left( \mathbb R / \mathbb Z \right)^2 $ into itself by
$(y,z)  \mapsto  \left( y+\frac{1}{\phi^2}  \mbox{ , }  z+y-\frac{1}{2\phi^3} \right)$,
is self-induced, minimal and uniquely ergodic. 
}
\begin{proof}[Proof of Proposition \ref{nil:th:nil}.]
Let us write explicitly how the matrix $\underline{\boldsymbol{g}}_\theta $ acts on the  fundamental domain $\boldsymbol{X}^s $.
\begin{figure}[H]
\centering
\includegraphics[width=8cm]{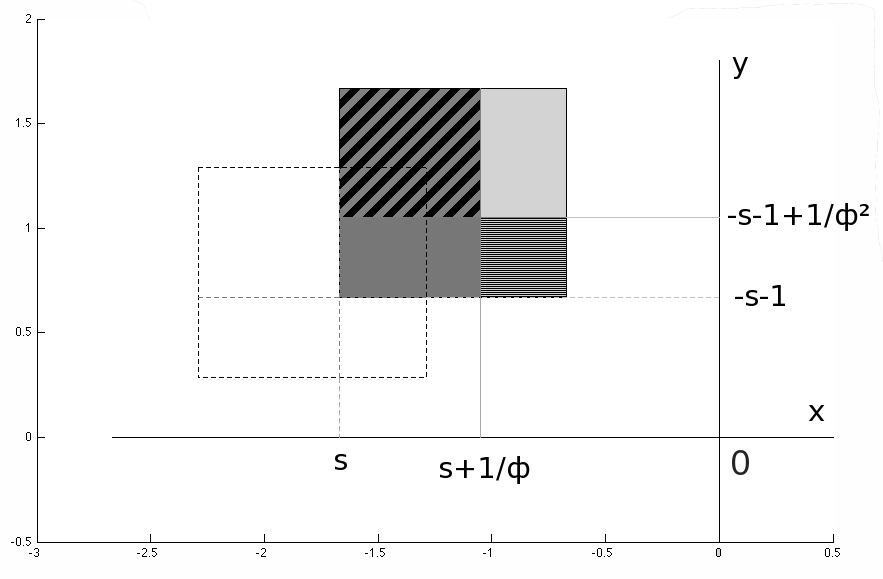}
\caption{Projection of the four areas of $\boldsymbol{X}^s$
which act on $\boldsymbol{g}_\theta$.}
\label{nil:fig:4zones}
\end{figure}
$$
\begin{array}{llllcc}
\mbox{If }
\left\{
\begin{array}{ll}
x\geq s+ \frac{1}{\phi}\\
 y \geq -s-1+\frac{1}{\phi^2}\\
 \end{array}
\right.
 &:&
 \boldsymbol{x} & \mapsto &
  \begin{bmatrix} x -\frac{1}{\phi} \\ y - \frac{1}{\phi^2 } \\ z - \frac{y}{\phi} + \theta \end{bmatrix} 
 \\
 \mbox{if }
\left\{
\begin{array}{ll}
x\leq s+ \frac{1}{\phi}\\
 y \geq -s-1+\frac{1}{\phi^2}
  \end{array}
\right.
 &:& 
 \boldsymbol{x} & \mapsto &
  \begin{bmatrix} x -\frac{1}{\phi}+1 \\ y - \frac{1}{\phi^2 } \\ z - \frac{y}{\phi} + \theta \end{bmatrix}
  \\
\end{array}
\begin{array}{cccccccc}
\mbox{, if }
\left\{
\begin{array}{ll}
x\geq s+ \frac{1}{\phi}\\
 y \leq -s-1+\frac{1}{\phi^2} \\
 \end{array}
\right.
  &:& 
 \boldsymbol{x} & \mapsto &
  \begin{bmatrix} x -\frac{1}{\phi} \\ y - \frac{1}{\phi^2 } + 1 \\ z +x -  \frac{y}{\phi} + \theta - \frac{1}{\phi} \end{bmatrix} 
  ,
  \\
\mbox{and if }
\left\{
\begin{array}{ll}
x\leq s+ \frac{1}{\phi} \\
 y \leq -s-1+\frac{1}{\phi^2}\\ 
 \end{array}
\right.
&:& 
 \boldsymbol{x} & \mapsto &
   \begin{bmatrix} x -\frac{1}{\phi} +1 \\ y - \frac{1}{\phi^2 } \\
    z + x - \frac{y}{\phi} + \theta + \frac{1}{\phi^2} \end{bmatrix}
 .  \\
\end{array}
$$
We fix parameters $(s,s',\theta)$ in $\mathbb R^3$ and we put:
$
\left\{\begin{array}{cccc}
 \mathcal S &  = & \{(y,z) \mbox{ such that } 0 \leq y \leq 1  \mbox{ and } 0\leq z \leq 1 \} ,\\
\mathcal S _{\mbox{\footnotesize Ind}}  & = &  \{(y,z) \mbox{ such that } 0  \leq y \leq  \frac{1}{\phi^2}    \mbox{ and } z \in[0,1] \}.
\end{array}\right.
$
\vspace{6pt}
\\
By ``forgetting'' for the moment, the first coordinated, the translation by $ \underline{\boldsymbol{g}} _\theta $
on $ \underline{\boldsymbol{X}}$ is conjugate to an application $T_{s}$ of $[-s-1,-s] \times [0,1]$ into itself defined by:
$$
\left\{ \begin{array}{ccccll}
T_s (y,z) & = &  \left( y- \frac{1}{\phi^2} + 1, z - \phi y +  \theta -  \frac{1}{\phi} \mbox{ mod }1 \right) &\mbox{if}& -s-1 \leq y \leq -s-1+ \frac{1}{\phi^2},\\
 T_s (y,z) & = &  \left(  y- \frac{1}{\phi^2} , z - \frac{y}{\phi }  + \theta  \mbox{ mod }1\right) &\mbox{if}& -s-1+  \frac{1}{\phi^2} \leq y \leq -s.
\end{array} \right.
$$ 
The application $T_s $ is conjugate by translation, to an application $ T ^ s $ of $ \mathcal S $ into itself defined by:
$$
\left\{
\begin{array}{ccccll}
T^s (y,z) & = &
 \left( y- \frac{1}{\phi^2} + 1, z - \phi y +  \theta -  \frac{1}{\phi} + (s+1)\phi \mbox{ mod }1 \right)
&\mbox{if}&
 0 \leq y \leq \frac{1}{\phi^2}.\\
 T^s (y,z) & = & 
\left(  y- \frac{1}{\phi^2} , z - \frac{y}{\phi }  + \theta + (s+1)/\phi \mbox{ mod }1\right)
&\mbox{if} &
 \frac{1}{\phi^2} \leq y \leq .
\end{array}
\right.
$$
We define the first return of the map $T^s$ of $\mathcal S_{\mbox{\footnotesize Ind}} $ into itself as follows: 
$$
T_{\mbox{\footnotesize Ind}} ^s (y,z) = {(T^s)}^{n_{y,z}} (y,z) \mbox{ where } n_{y,z}= \inf \Big{\{} n \in \mathbb N ^+ ; {(T^s)}^n(y,z) \in \mathcal S_{\mbox{\small Ind}}  \Big{\}}. 
$$
It is clear that $n_{y,z}=n_y$ only depends on $y$ (neither $z$, nor $s$, nor $\theta$), a simple calculation gives us: 
$$
n_y = 2 \mbox{ if $0 \leq y \leq  \frac{1}{\phi^4} $ and $n_y=3$ if $\frac{1}{\phi^4} \leq y \leq \frac{1}{\phi^2}$}.
$$
\begin{figure}[H]
\centering
\includegraphics[width=5.5cm]{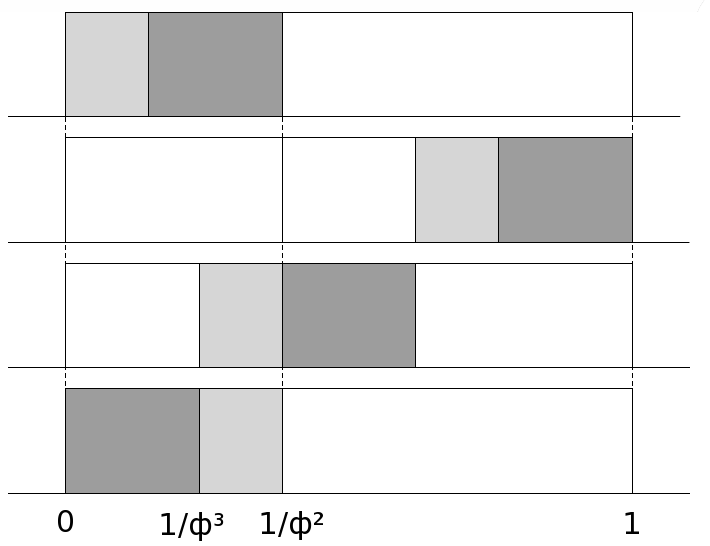}
\caption{Calculation of  $n_y$.}
\label{nil:fig:inductionor}
\end{figure}
\noindent
A direct calculation then yields the expression of $T_{\mbox{\footnotesize Ind}} ^s$:
$$
\begin{array}{llll}
T_{\mbox{\footnotesize Ind}} ^s(y,z) & = & \left(  y +\frac{1}{\phi^3},  z - y \left( \frac{1}{\phi} + \phi \right)  +2\theta + (s+1)\left( \frac{1}{\phi} + \phi \right) \mbox{ mod }1\right) \mbox{ if $0 \leq y \leq \frac{1}{\phi^4}$ },  \\
T_{\mbox{\footnotesize Ind}} ^s (y,z) & = & \left( y- \frac{1}{\phi^4},  z - y \left( \frac{2}{\phi} +  \phi \right) + 3 \theta - \frac{1}{\phi^4}  + (s+1)\left( \frac{2}{\phi} + \phi \right)\mbox{ mod }1 \right) \mbox{ if $ \frac{1}{\phi^4} \leq y \leq \frac{1}{\phi^2}$ .}
\end{array}
$$
We consider the application $\Phi$ from $\mathcal S _{\mbox{\footnotesize Ind}}$ into $\mathcal S$:
$$
\Phi  (y,z) = (\phi^2 y , a y^2 + by +z \mbox{ mod } 1) \mbox{ and } {\Phi} ^{-1} (y,z) = (\phi^{-2} y , z - \frac{a}{\phi^4} y^2 - \frac{b}{\phi^2} y \mbox{ mod } 1 ).
$$
The application $\underline{T}^s = \Phi \circ T_{\mbox{\footnotesize Ind}}  \circ \Phi ^{-1} $ of $\mathcal S $ into itself,  obtained by the transfer function $\Phi$ is:
\vspace{6pt}
\\
\underline{If $ 0 \leq y \leq \frac{1}{\phi^2}$}:
$$
\begin{array}{llll}
\underline{T}^s (y,z) & = & \Phi \circ T_{\mbox{\footnotesize Ind}} ^s (\phi^{-2} y , z - \frac{a}{\phi^4} y^2 - \frac{b}{\phi^2} y  \mbox{ mod }1)
\\
& =& \Phi \circ \mbox{\Large{(}}  \phi^{-2} y  - \frac{1}{\phi^4} + \frac{1}{\phi^2}, z - \frac{a}{\phi^4} y^2 - \frac{b}{\phi^2} y 
- \phi^{-2} y  \left( \frac{1}{\phi} + \phi \right)  
 + 2 \theta + (s+1)\left( \frac{1}{\phi} + \phi \right) \mbox{ mod }1  \mbox{\Large{)}}  
\\
& =&  \mbox{\Large{(}}  y  - \frac{1}{\phi^2} + 1, z - \frac{a}{\phi^4} y^2 - \frac{b}{\phi^2} y 
- \phi^{-2} y  \left( \frac{1}{\phi} + \phi \right) + 2 \theta  + (s+1)\left( \frac{1}{\phi} + \phi \right)
 \\
& & +a (  \phi^{-2} y  - \frac{1}{\phi^4} + \frac{1}{\phi^2}  )^2+ b  (\phi^{-2} y  - \frac{1}{\phi^4} + \frac{1}{\phi^2})
 \mbox{ mod }1 \mbox{\Large{)}} 
 \\
 & =&  \left( y  - \frac{1}{\phi^2} + 1, z 
- \phi^{-2} y  \left( \frac{1}{\phi} + \phi -\frac{2a}{\phi^3} \right) 
+ \frac{b}{\phi^3}  
 + \frac{a}{\phi^6} + 2 \theta  + (s+1)\left( \frac{1}{\phi} + \phi \right)
 \mbox{ mod }1 \right),
\end{array}
$$
\underline{and if $\frac{1}{\phi^2}\leq y \leq 1$}:
$$
\begin{array}{llll}
\underline{T}^s (y,z) & = &
\Phi \circ T_{\mbox{\footnotesize Ind}} ^s (\phi^{-2} y , z - \frac{a}{\phi^4} y^2 - \frac{b}{\phi^2} y  \mbox{ mod }1)
\\
& =& \Phi \circ  \mbox{\Large{(}}  \phi^{-2} y  - \frac{1}{\phi^4}, z - \frac{a}{\phi^4} y^2 - \frac{b}{\phi^2} y 
- \phi^{-2} y  \left( \frac{2}{\phi} + \phi \right) 
+ 3 \theta- \frac{1}{\phi^4} +  (s+1)\left( \frac{2}{\phi} + \phi \right)  \mbox{ mod }1 \mbox{\Large{)}}
 \\
& =&   \mbox{\Large{(}}  y  - \frac{1}{\phi^2}, z - \frac{a}{\phi^4} y^2 - \frac{b}{\phi^2} y 
- \phi^{-2} y  \left( \frac{2}{\phi} + \phi \right) + 3 \theta-
\frac{1}{\phi^4} +  (s+1)\left( \frac{2}{\phi} + \phi \right)\\
& & +a (   \phi^{-2} y  - \frac{1}{\phi^4} )^2+ b  (  \phi^{-2} y  - \frac{1}{\phi^4} )
 \mbox{ mod }1  \mbox{\Large{)}} \\
 & = & \left(  y  - \frac{1}{\phi^2}, z
- \phi^{-2} y  \left( \frac{2}{\phi} + \phi + \frac{2a}{\phi^4} \right) +
3 \theta - \frac{1}{\phi^4} + \frac{a}{\phi^8}-\frac{b}{\phi^4} +  (s+1)\left( \frac{2}{\phi} + \phi \right)
 \mbox{ mod }1 \right) .
\end{array}
 $$ 
To get the desired result, we must find $ \theta '$  such that the function $ \underline{T} ^s $  belongs to the family of initial functions. 
Therefore, select parameters  $a$ and $ b $ such that we can find a $\theta '$  such that the following two systems admit a solution: 
$$
\mbox{ (S$_1$) : }
\left\{ \begin{array}{cccc}
-\phi & = & - \phi^{-2}  \left( \frac{1}{\phi} + \phi -\frac{2a}{\phi^3} \right) ,\\
-\frac{1}{\phi} & = &- \phi^{-2}   \left( \frac{2}{\phi} + \phi + \frac{2a}{\phi^4} \right) ,\\
\end{array} \right.
$$
$$
\mbox{ and (S$_2$) : }
\left\{  \begin{array}{cccc} 
\theta' -\frac{1}{\phi} + (s'+1)\phi &= &  \frac{b}{\phi^3}   + \frac{a}{\phi^6} + 2 \theta + (s+1)\left( \frac{1}{\phi} + \phi \right) ,\\
\theta ' + (s'+1)/\phi & = & 3 \theta - \frac{1}{\phi^4} + \frac{a}{\phi^8}-\frac{b}{\phi^4}+  (s+1)\left( \frac{2}{\phi} + \phi \right) .\\
 \end{array} \right. 
$$
The first system has a unique solution : $ a = -\phi^3. $
\vspace{6pt}
\\
The parameters $ b $ and $ s '$ are related by: $\frac{1}{\phi} - (s'+1) = \theta -\frac{1}{\phi^2} b + (s+1) \frac{1}{\phi}$, which fixes the value of parameters:
$$
b = \phi^2 \theta + \phi(s+1) + \phi^2 (s'+1) -\phi.
$$
So we  find: 
$\theta'   =   \frac{b}{\phi^3}   + \frac{a}{\phi^6} + 2 \theta + (s+1)\left( \frac{1}{\phi} + \phi \right) + \frac{1}{\phi} - (s'+1)\phi
= \phi^2 \theta + \phi^2 (s+1) - (s'+1).$
\end{proof}

\section{Proof of Theorem \ref{th7}}
\label{nil:se:flot}

\noindent
We start by fix a nilflow periodic under renormalization and its associate automorphisms $\mathcal L$.
By Proposition \ref{nil:prop:autosubs}, there exists $\sigma$, an automorphism on $\mathbb F_2$, defined in (\ref{nil:eq:substitution}), such that $\mathcal L = \mathfrak S_\sigma$.
The periodic points of the renormalization flow of L. Flaminio and G. Forni, are semi-simple hyperbolic automorphisms
which stabilize the discrete Heisenberg group $\Gamma$, and preserve the center, up to a change of orientation.
So, $\sigma$ is a \textbf{hyperbolic}, \textbf{unimodular} automorphism,
that is to say that we impose on $M_\sigma$ hypothesis $(H)$:

$$
\mbox{$M_\sigma= \begin{pmatrix} A & B \\ C & D \end{pmatrix}$
admits two reals eigenvalues $\lambda$ and $\lambda'$ such that $\mid \lambda \lambda ' \mid = 1$
and $\mid \lambda \mid > \mid \lambda' \mid$.}
\ \ \ \ \ \ \ \ 
(\boldsymbol{H})
$$
Let $(\alpha, \beta)$ be a nontrivial eigenvector 
associated to the eigenvalue $\lambda$.
We can interchange ``$a$" with ``$a^{-1}$", or ``$b$" with ``$b^{-1}$", so that we can choose 
$\alpha\geq 0$ and $\beta \geq 0$, such that $\alpha +\beta =1$.
In this way, if $v$ is an infinite word on $\{a,b\}$
such that $\sigma(v)=v$, then $\alpha$  corresponds exactly to the frequency of occurrence of 
``$a$'' in $v$, and $\beta$ to the frequency of occurrence of symbol ``$b$''.
\vspace{6pt}
\\
We fix $(\alpha',\beta')$,  a nontrivial eigenvector 
associated to the eigenvalue $\lambda'$.
Recall that under these conditions, we have: 
$\alpha' \beta' < 0$, $\lambda \notin \mathbb Q$ and $\lambda' \notin \mathbb Q$.
To simplify the calculations, we impose that ${\alpha'}^2+{\beta'}^2=1$.
\vspace{6pt}
\\
According to hypothesis $(\textbf{H})$, the values $\alpha'$ and $\beta'$
are nonzero. We fix $\alpha'$ strictly negative. 
We put $\Delta = \alpha \beta ' -\alpha' \beta \neq 0$
and we write $\Phi _\lambda $ and  $\Phi _{\lambda'} $ the flows
obtained by Proposition \ref{nil:prop:flotcomm}.
We write $\gamma$ and $\gamma'$ the reals given by equation (\ref{nil:eq:gamma}).
We note that $\Delta$ is negative because 
$\beta' \Delta = (\beta ')^2 \alpha + (-\beta')\alpha' \beta >0.$
\vspace{6pt}
\\
The flows $\Phi _\lambda $ and  $\Phi _{\lambda'} $
generate a surface
$
S = 
\left\{ \Phi _\lambda ^t \circ \Phi_{ \lambda '} ^s( \boldsymbol{0} ); (t,s)\in \mathbb R ^2 \right\}
\mbox{ and we write $x_{t,s} =  \Phi _\lambda ^t \circ \Phi_{\lambda '} ^s( \boldsymbol{0} )$}.
$
\vspace{6pt}
\\
Let $t_a$ and $t_b$ be the reals defined by 
$
\left\{ \begin{array}{cccc}  (t_a \alpha -1) \beta' & = & t_a \beta \alpha' , \\  t_b \alpha \beta ' & = & (t_b \beta-1) \alpha' , \end{array} \right.
\Longleftrightarrow
\left\{ \begin{array}{cccccc} t_a  & = & \frac{\beta' }{  \beta' \alpha  - \alpha ' \beta} & = &  \frac{\beta' }{ \Delta} >0 ,\\  t_b  & = & \frac{-\alpha' }{  \beta' \alpha  - \alpha' \beta} & = & \frac{ - \alpha' }{ \Delta}>0. \\ \end{array} \right. 
$
\vspace{6pt}
\\
We put 
$ d_a   =  \sqrt{ (t_a \alpha-1)^2 + (t_a \beta)^2 }$ and
$  d_b   =  \sqrt{ (t_b \alpha)^2 + (t_b \beta-1)^2 }$.
We write $\mathcal D = \mathcal D_a \cup \mathcal D_b$, where:
$$
\mathcal D _a   =  \left\{ (\alpha ' s+t \alpha,  \beta' s + t \beta);
-d_a  \leq s < 0 \mbox{ and } 0 \leq t  < t_b \right\} \mbox{ and }
 \mathcal D _b   =  \left\{ (\alpha ' s+t \alpha,  \beta' s + t \beta);
  0 \leq s < d_b \mbox{ and } 0 \leq t  < t_a \right\}.
  $$
\begin{figure}[H]
\centering
\includegraphics[width=6.5cm]{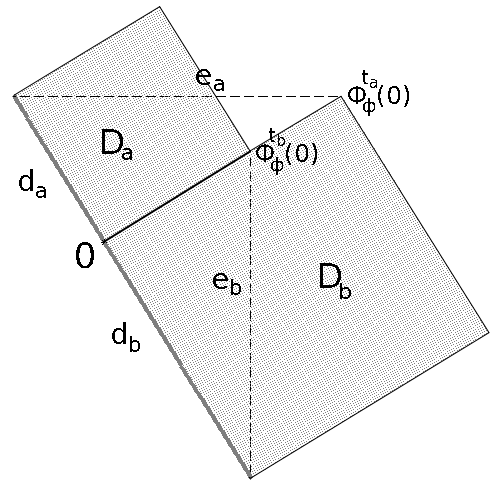}
\caption{Domain $D_\tau$ associated to the substitution of Fibonacci.}
\label{nil:fig:domaineor}
\end{figure}
\begin{proposition}
 \label{nil:prop:eqsurf}
There exists a polynomial $ Q_\sigma $ of degree $ 2 $ in $ x $ and $ y $ 
such that $\boldsymbol{x} = [x,y,z] \in S$ if and only if $z=Q_\sigma(x,y)$.
 \end{proposition}
\noindent
By Proposition \ref{nil:prop:flotcomm}, $\mathfrak S$ acts on $S$ by: 
$\mathfrak S (x_{t,s}) = \mathfrak S \circ \Phi _\lambda ^t \circ \Phi_{ \lambda'} ^s( \boldsymbol{0} )
= \Phi _\lambda ^{\lambda t} \circ \Phi_{ \lambda '} ^{ \lambda '}( \boldsymbol{0} )
= x_{\lambda t, s  \lambda '}.$
\vspace{6pt}
\\
We consider the ``\textbf{tile}'':
$
\mathcal T = \left\{ 
\begin{bmatrix} x  \\ y \\ z  \end{bmatrix} \in \boldsymbol{X} ; (x,y) \in \mathcal D \mbox{ and } 
Q_\sigma(x,y) -1/2 \leq z <  Q_\sigma(x,y)+1/2
\right\}. 
$
\begin{proposition}
\label{nil:prop:dfflot}
$\mathcal T$ is a \textbf{fundamental domain} of $ \boldsymbol{X}$.
\end{proposition}

\noindent
The proof of Proposition \ref{nil:prop:dfflot} is a direct conscequence of the fact that $D_\sigma$ is a fondamental domain for $\mathbb R^2$ (\cite{Adler}).
The aim is to consider the properties of the first return flow in a  "good" section. This section will be the surface $\Sigma$ defined below.
Proposition \ref{nil:prop:autoind} ensures us that this application is  self-induced. Then we will see in Proposition \ref{nil:th:niltran},
 that this application is conjugated to a niltranslation, which will have property also 
to be self-induced. Proposition \ref{nil:prop:alafin} assures us that this
niltranslation is minimal and uniquely ergodic on 
a surface isomorphic to the torus ${(\mathbb S^1)}^2$.
It will complete the proof of Theorem \ref{th7}.

\begin{figure}[H]
\centering
\includegraphics[width=7cm]{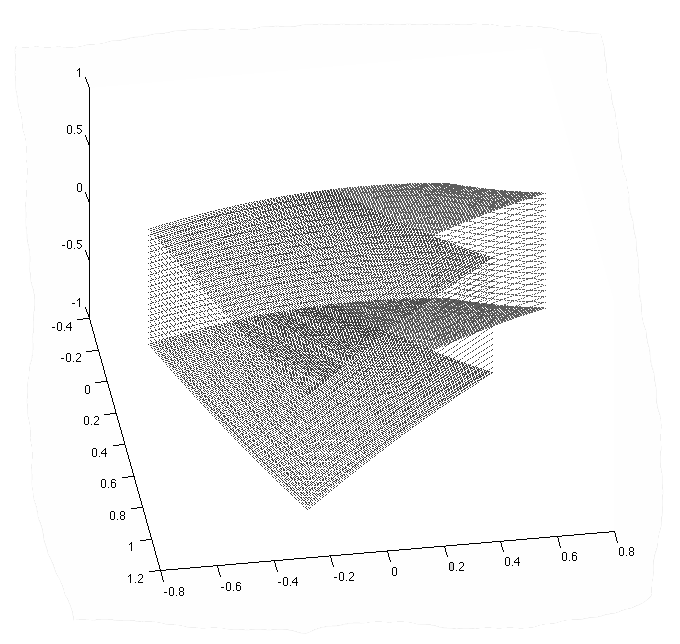}
\includegraphics[width=7cm]{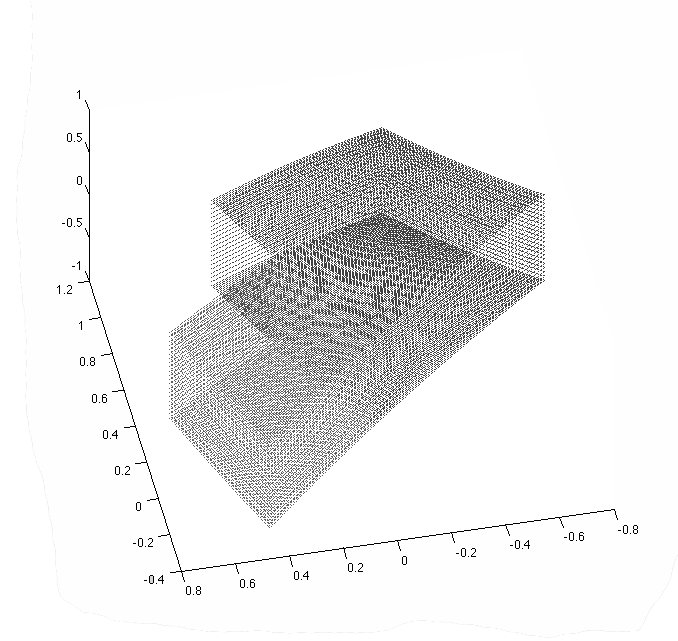}
\caption{Two angles of the tile associated to the Fibonacci substitution.}
\label{nil:fig:tuileor}
\end{figure}
\noindent
Consider the following section: 
$
\Sigma = \left\{ 
\begin{bmatrix} \alpha ' s   \\  \beta' s \\ z  \end{bmatrix} ; s  \in 
[- d_a,d_b [ \mbox{ and } 
Q_\sigma( \alpha ' s  ,  \beta' s ) -\frac{1}{2} \leq z <  Q_\sigma( \alpha ' s  ,  \beta' s )+\frac{1}{2}
\right\} .
$
\vspace{6pt}
\\
We denote by $T_{\Sigma}$ the application of first return of flow $\Phi_\lambda$ from the section $\Sigma$ into itself. 
\begin{proposition}
\label{nil:prop:autoind}
The application $T_{\Sigma}$ is self-induced. 
\end{proposition}

\begin{proof}
Let $\boldsymbol{x}$ be a point of  $\mathfrak S(\Sigma)$ and
$t_{\boldsymbol{x}} = \inf  \left\{ t >0; \Phi_\lambda ^t (\boldsymbol{x}) \in \mathfrak S(\Sigma) \right\} \in [0,+ \infty]$.
\vspace{6pt}
\\
We denote by $T_{\mathfrak S(\Sigma)}$ the application of $\mathfrak S(\Sigma)$ into itself defined by $T_{\mathfrak S(\Sigma)} (\boldsymbol{x})  = \Phi_\lambda ^{t_{\boldsymbol{x}} } (\boldsymbol{x}).$
According to the sign of $\det M_\sigma$,  two situations may occur. In all cases,  we have
$t_{\boldsymbol{x}}  \in \{  \min(\lambda t_a,\lambda t_b) ,\max(\lambda t_a,\lambda t_b) \}$.
Assume in the remainder of the proof, that  $\det(M_\sigma)>0$.
\vspace{6pt}
\\
Let $\boldsymbol{x} =\begin{bmatrix} x \\ y \\ z \end{bmatrix}$,
by Proposition \ref{nil:prop:flotcomm}:
$$
\left\{
\begin{array}{llll}
\mbox{ if $x>0$} , \
\mathfrak S^{-1} \circ T_{\mathfrak S(\Sigma)} \circ \mathfrak S (\boldsymbol{x}) =
\mathfrak S^{-1} \circ \Phi_\lambda ^{\lambda t_{\boldsymbol{a}}} \circ \mathfrak S (\boldsymbol{x})=
 \Phi_\lambda ^{\frac{1}{\lambda} \lambda t_{\boldsymbol{a}}} (\boldsymbol{x})
 =
 \Phi_\lambda ^{ t_{\boldsymbol{a}} } (\boldsymbol{x}) = T_\Sigma (\boldsymbol{x} ),
\\
\mbox{ if $x<0$} , \
\mathfrak S^{-1} \circ T_{\mathfrak S(\Sigma)} \circ \mathfrak S (\boldsymbol{x})  = 
\mathfrak S^{-1} \circ \Phi_\lambda ^{\lambda t_{\boldsymbol{b}}} \circ \mathfrak S (\boldsymbol{x})=
 \Phi_\lambda ^{\frac{1}{\lambda} \lambda t_{\boldsymbol{b}}} (\boldsymbol{x})
 = T_\Sigma (\boldsymbol{x} ).
\end{array}
\right.
$$
\end{proof}
\noindent
We associate to $\sigma$,
$\boldsymbol{\mathfrak{x}}_\sigma =  \begin{pmatrix} \alpha \\ \beta \\ \gamma  \end{pmatrix} \in \boldsymbol{\mathfrak{g}}$,
the matrix $\boldsymbol{g}_\sigma = \exp \boldsymbol{\mathfrak{x}}_\sigma \in \boldsymbol{X}$
and the niltranslation 
$T_\sigma : 
\begin{array}{cccc}
\boldsymbol{X} & \longrightarrow & \boldsymbol{X} \\
\boldsymbol{x} & \mapsto &  \boldsymbol{g}_\sigma \bullet \boldsymbol{x}
\end{array}.
$
\vspace{6pt}
\\
We consider the surfaces: 
$ \mathcal D = \left\{ \begin{bmatrix}x \\ y \\ z \end{bmatrix} ; x+y \in \mathbb Z \right\} \mbox{ and } D= \left\{ \begin{bmatrix}x \\ y \\ z \end{bmatrix} ; x+y = 0 \right\}$.

\begin{proposition}
\label{nil:th:niltran} 
The surface $\mathcal D $ immersed in the quotient space, denoted $\underline{ \mathcal D}$,
is a section of the flow $\Phi_\lambda$ with a return time constant, equal to $1$.
The application of first 
return flow in this section coincides with the niltranslation $\underline{T_\sigma}$
on $\underline{\mathcal D}$ .
In addition, the application $T_\Sigma$ on $\Sigma$ is measurably conjugate to the 
niltranslation $\underline{T_\sigma}$ on $\underline{\mathcal D}$
which is also self-induced. 
\end{proposition}
\begin{proof}
The goal is to construct a bijection $ \psi $ bi-measurable between the sections 
 $\underline{ \mathcal D}$ and $\Sigma$.
\begin{figure}[H]
\centering
\includegraphics[width=7.5cm]{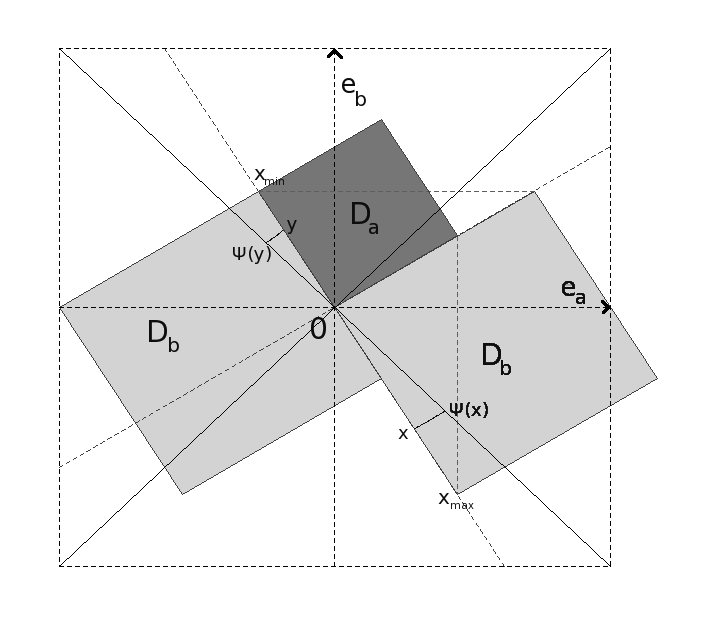}
\includegraphics[width=7.5cm]{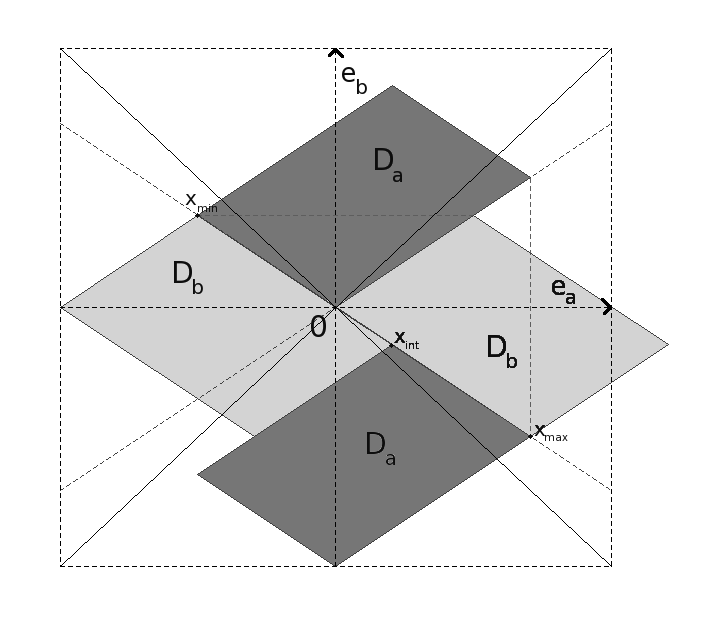}
\caption{Represation of the tiles in the case where  $\alpha'\leq -\beta'$, and $\alpha'>-\beta'$.}
\label{nil:fig:demo}
\end{figure}
\noindent
We treat only the case where $\alpha'\leq -\beta'$, the other case being analogous. We assume 
$\alpha' + \beta'<0$.
Let 
$$
\boldsymbol{x} = \begin{bmatrix} \alpha' s \\ \beta ' s \\ Q_\sigma(\alpha' s,\beta' s) \end{bmatrix} \in \Sigma
\mbox{ with } 
s \in \left[-d_a ,  d_b \right) .
$$
We define the time $t^{\boldsymbol{x}}$, so that $\Phi^{t^{\boldsymbol{x}}}_\lambda (\boldsymbol{x}) \in D$ satisfies :
$
\alpha' s + t^{\boldsymbol{x}} \alpha = - \left( \beta' s + t^{\boldsymbol{x}} \beta \right)
\mbox{ , so } t^{\boldsymbol{x}} = 
t^{\boldsymbol{x}} (\alpha + \beta)
=
-(\beta' + \alpha ')s.
$
\mbox{}
\\
The key of the demonstration is to verify that the flow $\Phi_\lambda$,
beginning at a point $ \boldsymbol{x} $, does not intersect the surface $ \Sigma $ before 
intersecting the diagonal surface $ D $. For this, 
\textit{we have just to check the following two conditions:} 
\begin{eqnarray}
\label{nil:eq:flotnil1}
\mbox{ If $\boldsymbol{x}_{\mbox{\footnotesize max}} $}
& = & 
\mbox{$ \begin{bmatrix} \alpha' d_b
 \\ \beta ' d_b \\ 0 \end{bmatrix}$,
then }
 t^{\boldsymbol{x}_{\mbox{\footnotesize max}}} < t_b,\\
\label{nil:eq:flotnil2}
\mbox{ and if $\boldsymbol{x}_{\mbox{\footnotesize min}} $}
&= &
\mbox{$
\begin{bmatrix} - \alpha ' d_a
 \\ -\beta ' d_a  \\ 0 \end{bmatrix}$,
then }
 - t^{\boldsymbol{x}_{\mbox{\footnotesize min}}} < t_b.
\end{eqnarray}
We begin by verifying equation (\ref{nil:eq:flotnil1}).
$$
\begin{array}{llll}
 t^{\boldsymbol{x}_{\mbox{\footnotesize max}}} < t_b
 & \Longleftrightarrow &  -(\beta'+ \alpha')d_b < t_b 
   \Longleftrightarrow   -\sqrt{(t_b \alpha)^2+(t_b\beta-1)^2} ( \beta'+ \alpha'  ) < t_b \\
   & \Longleftrightarrow &  -\sqrt{(\frac{-\alpha'}{\Delta} \alpha)^2+( \frac{-\alpha'}{\Delta}\beta-1)^2}   ( \beta'+ \alpha'  ) < \frac{\beta'}{\Delta} \\
   & \Longleftrightarrow &  -\mid \frac{1}{\Delta} \mid \sqrt{(\alpha' \alpha)^2+( -\alpha' \beta-\Delta)^2} ( \beta'+ \alpha'  )  < \frac{\beta'}{\Delta} \\
   & \Longleftrightarrow &   \sqrt{(\alpha' \alpha)^2+( \alpha \beta')^2}  ( \beta'+ \alpha'  )   > \beta' 
    \Longleftrightarrow    \alpha ( \beta'+ \alpha' ) > \beta' .
   \end{array}$$
\textit{Thus, it is sufficient to verify that:}
$
\beta' < \alpha \beta ' < \alpha \beta' + \alpha \alpha'
\mbox{\textit{ as $\alpha>0$, $\alpha'>0$ and $\beta'<0$.}}
$
\vspace{6pt}
\\
We now prove that equation (\ref{nil:eq:flotnil2}) is satisfied: 
$$
\begin{array}{llll}
- t^{\boldsymbol{x}_{\mbox{\footnotesize min}}} < t_b
 & \Longleftrightarrow &  -(\beta'+ \alpha') (- d_a)  < t_b 
   \Longleftrightarrow   \sqrt{(t_a \alpha -1)^2+(t_a \beta)^2} ( \beta'+ \alpha'  )< t_b \\
   & \Longleftrightarrow &  \sqrt{(\frac{\beta'}{\Delta} \alpha-1)^2+( \frac{\beta '}{\Delta}\beta)^2} ( \beta'+ \alpha'  )  < \frac{\beta'}{\Delta} \\
   & \Longleftrightarrow &  \mid \frac{1}{\Delta} \mid \sqrt{(\alpha\beta'-\Delta)^2+( \beta' \beta)^2}( \beta'+ \alpha'  )  < \frac{\beta'}{\Delta} \\
   & \Longleftrightarrow &   \sqrt{(\alpha' \beta)^2+( \beta \beta')^2} ( \beta'+ \alpha'  )  > \beta'     \Longleftrightarrow    \beta ( \beta'+ \alpha' ) > \beta' .
   \end{array}
$$
We conclude as before.  The case $ \alpha '+ \beta'> 0 $ is treated similarly.  We must verify the relations: 
$$
\mbox{ If } 
\boldsymbol{x}_{\mbox{\footnotesize min}}  = \begin{bmatrix} -\alpha ' d_a  \\ -\beta ' d_a  \\ 0 \end{bmatrix}
 \mbox{ , }
\boldsymbol{x}_{\mbox{\footnotesize int}}  =   \begin{bmatrix} \alpha' (d_b-d_a) \\ \beta ' (d_b-d_a) \\ 0 \end{bmatrix}
\mbox{ and } 
\boldsymbol{x}_{\mbox{\footnotesize max}}  =  \begin{bmatrix} \alpha' d_b \\ \beta ' d_b \\ 0 \end{bmatrix},
$$
then:
$t^{\boldsymbol{x}_{\mbox{\footnotesize min}}} < t_a$, $ - t^{\boldsymbol{x}_{\mbox{\footnotesize int}}} < t_a$ and $ -t^{\boldsymbol{x}_{\mbox{\footnotesize max}}} < t_b$.
The map $\psi : \Sigma \longrightarrow D$, defined by $\psi(\boldsymbol{x})= \Phi_\lambda ^{ t^{\boldsymbol{x}}}(\boldsymbol{x})$
is therefor a bijection between $\Sigma$ and $\psi(\Sigma)$. By construction, $\underline{T}_\sigma = \psi \circ {T}_\Sigma \circ \psi^{-1}$.
\end{proof}

\noindent
We thus arrive at the following result: 

\begin{proposition}
\label{nil:prop:alafin}
The niltranslation $\underline{T}_\sigma$ on $\underline{ \mathcal D}$ is self-induced, minimal and uniquely ergodic. 
\end{proposition}

\begin{proof}
We have already seen in the previous theorem that this application is self-induced.  To prove the result, by Theorem \ref{th:emmanuelnil} of E. Lesigne, 
we just have to check that this map is uniquely ergodic.  Since this mapping is continuous, we know that there is at least one invariant measure. 
Suppose there are two, denoted $\mu_1$ and $\mu_2$. We choose them disctinct and ergodic.
The map $\psi$ transports this measure into two different measure $\mu_1 ^\star$ and $\mu_2 ^\star$ on $\Sigma$.
We denote by $\Sigma_a$ (respectively $\Sigma _b$), the set of elements  $\boldsymbol{x}$ in $\Sigma$
such that $\boldsymbol{p}(\boldsymbol{x})\in \mathcal D_a$ (respectively  $\boldsymbol{p}(\boldsymbol{x})\in \mathcal D_b$).
We define two singular measures  $\nu_1$ and $\nu_2$ on $\mathcal T$ as follows. For any continuous
function $f$ on $\mathcal T$, we define a function  $\widetilde f$, measurable on $\Sigma$ by :

$$
\widetilde{ f} (\boldsymbol{x}) = \int \limits_0 ^{t_b} f \left( \Phi_\lambda ^t (\boldsymbol{x}) \right) dt
\mbox{ if } \boldsymbol{x} \in \Sigma_a,
\mbox{ and }
\widetilde{ f} (\boldsymbol{x})  = \int \limits_0 ^{t_a} f \left( \Phi_\lambda ^t (\boldsymbol{x}) \right) dt
\mbox{ if } \boldsymbol{x} \in \Sigma_b.
$$
The measures $\nu_1$ and $\nu_2$
are then defined by: 
$\nu_1(f) = \mu_1 ^\star (\widetilde f)
\mbox{ and }
\nu_2(f) = \mu_2 ^\star (\widetilde f).$
\vspace{6pt}
\\
These measures are invariant 
by the action of the flow $\Phi_\lambda$. To conclude, we only have to verify that 
this flow is uniquely ergodic.
Theorem \ref{nil:th:AGH}
of L. Auslander, L. Green and F. Hahn
assures us that it suffices to show that the ratio $\frac{\alpha}{\beta}$
is irrational. 
If this ratio is rational, this implies that the eigenvalue
$\lambda = m^{a,a}+ m^{a,b} \frac{\alpha}{\beta}$ is itself rational, 
which is absurd since it is a root of an irreducible polynomial of degree $2$
in $\mathbb Z [X]$.
\end{proof}


\section{More about self-induction}
\label{se:fin}
\noindent
We will now consider to a partial converse of Theorem \ref{th7}, and we will see that there is an obvious obstruction.
\vspace{6pt}
\\
We fix a matrix $M= \begin{pmatrix} A & B \\ C & D \end{pmatrix}$, $\alpha$ and $\beta$ as in Theorem \ref{th7}.
We put $\gamma_0 =  -  \frac{ \alpha A C+ \beta B D }{2\lambda -2 \det(M)}$.
For any $x\in \mathbb R$, we defined a homeomorphism $C_x$ of $ \underline{\boldsymbol{X}}$ by 
$ C_x (\boldsymbol{y} \bullet \boldsymbol{\Gamma} )= \begin{bmatrix} x \\ 0 \\ 0 \end{bmatrix} \bullet  \boldsymbol{y} \bullet \boldsymbol{\Gamma}$.
\vspace{6pt}
\\
For any $x\in \mathbb R$, the niltranslations by $\boldsymbol{g}=\begin{bmatrix} \alpha \\ \beta \\ \gamma + \frac{\alpha\beta}{2} \end{bmatrix} $
and  $\boldsymbol{g}(x)=\begin{bmatrix} \alpha \\ \beta \\ \gamma + \frac{\alpha\beta}{2} + \beta x \end{bmatrix} $ are conjugate via $C_x$.
We fix $x_0$ such that for any $(n,m)\in \mathbb Z^2$,
$\gamma + \beta x_0 \neq    \frac{ \alpha}{\lambda - \det(M)}  \left(  n - \frac{AC}{2} \right) +  \frac{ \beta }{\lambda - \det(M)}  \left( m - \frac{BD}{2} \right)$.
From the work developed in Section \ref{nil:se:flot}, the niltranslation by 
$\boldsymbol{g}(x_0)$
is self-induced, and is  the return  map of a nilflow not periodic under renormalization.
\vspace{6pt}
\\
In the proof of Proposition \ref{nil:th:niltran},  we explicity constructed the renormalization map.
A serious problem of our work is that unlike the abelian case, this application is not a morphism.
However, we believe that there is a partial converse of the theorem, but it is difficult to imagine what kind of renormalizations involved.
\vspace{6pt}
\\
We conclude by constructing an example of self-induced niltranslation, for which the areas of induction does not project well on abelianisation.
\vspace{6pt}
\\
To simplify the notation, as we have seen in the section, we will be interested by the application
$T_\phi$, from $\mathbb R^2/\mathbb Z^2$ into itself, defined by
$T_\phi(x,y) =( \ x+1/\phi^2,y+x-1/(2\phi^3) \ )$. 
We saw in Section \ref{nil:se:ex} that it is self-induced, and that it was equivalent to consider this application, or the associated niltranslation.
We consider the quadratic functions $ p $, $ q $ and $ r $, defined for all real $ x $ by:
$$
p (x)=\phi^2x^2/2-\phi x/2-1/\phi, \
q (x)=p(x)+\phi^2x +3/2
\mbox{ and }r (x)=p(x)-\phi^2x +1+1/(2\phi^3).
$$
We defined two aeras  $D_1$ and  $D_2$ by:
$$
\begin{array}{clclcl}
& &D_1 & = & \Big{\{} \ (x,y)\ ;\ p (x)<y\leq p (x)+1 \mbox{ and } y \leq \min(q (x),r (x)-1) \  \Big{\}} \\
& \mbox{ and }& D_2 & = &  \Big{\{} \ (x,y)\ ;\ p (x)<y\leq p(x)+1 \mbox{ and } r (x)-1<y \leq r (x) \  \Big{\}}  .
\end{array}
$$
We define $R$, the application from $D=D_1\cup D_2$ into itself by  $R (\boldsymbol{x})=T_\phi(\boldsymbol{x})$,
if $\boldsymbol{x}\in D_1$, and  $R (\boldsymbol{x})=T_\phi(\boldsymbol{x})-(1,0)$
if $\boldsymbol{x}\in D_2$. 
The dynamical system engendered by $R$ is topologically conjugate to that engendered by $T_\phi$.
\begin{figure}[H]
\centering
\includegraphics[width=12cm]{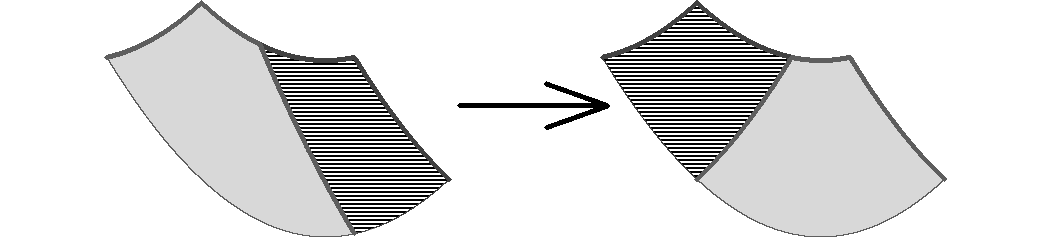}
\caption{Representation of $R$.}
\label{fig:e1}
\end{figure}
\begin{proposition}
The first return map of $R$ into $D_2$, is conjugate to $R$.
\end{proposition}
\begin{proof}
We begin by conjugate $R$ with an application defined on an isosceles trapezoid, vie 
$\psi(x,y)=(x,y-p(x))$.
We fix
$$
\begin{array}{clclcl}
& &D_1 '& = & \Big{\{} \ (x,y)\ ;\ 0 <y \leq 1 \mbox{ and } y \leq \min(\phi^2x +3/2,-\phi^2x +1/(2\phi^3) ) \  \Big{\}} \\
& \mbox{ and }& D_2' & = &  \Big{\{} \ (x,y)\ ;\ 0<y\leq 1 \mbox{ and }  -\phi^2x +1/(2\phi^3)  <y \leq -\phi^2x +1+1/(2\phi^3) \  \Big{\}}  .
\end{array}
$$
We define $R'$, the application from $D'=D_1'\cup D_2'$ into itself by  $R' (x,y)=R'_1(x,y)=(x+\frac{1}{\phi^2},y)$,
if $(x,y)\in D_1'$, and  $R' (x,y)=R'_2(x,y)=(x+1/\phi^2-1,y+\phi^2 x-1/(2\phi^3))$ if $(x,y) \in D_2'$. 
\begin{figure}[H]
\centering
\includegraphics[width=12cm]{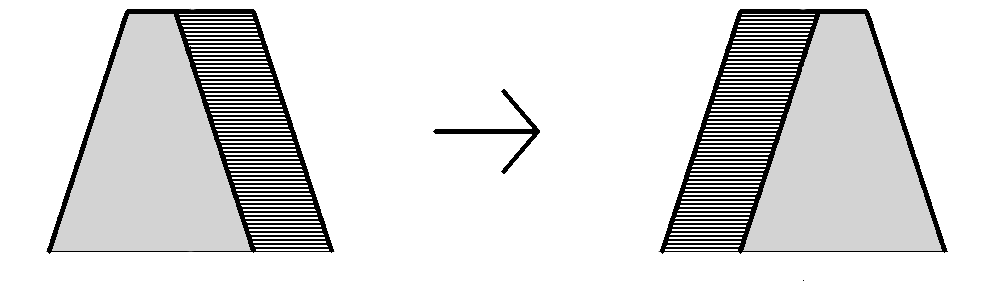}
\caption{Representation of $R'$.}
\label{fig:im}
\end{figure}
\noindent
The first return map of $R$ into $D_2$, is conjugate to the first return map of $R'$ into $D_2'$.
We can compute that for any $\boldsymbol{x}$ in $D_2'$, it exists an integer
$n_{\boldsymbol{x}} \in \{0,1,2,3\}$ such that the first return map of $R'$ into $D_2'$
is equal to ${R_1'}^n \circ R_2'(x,y)$.
\begin{figure}[H]
\centering
\includegraphics[width=9cm]{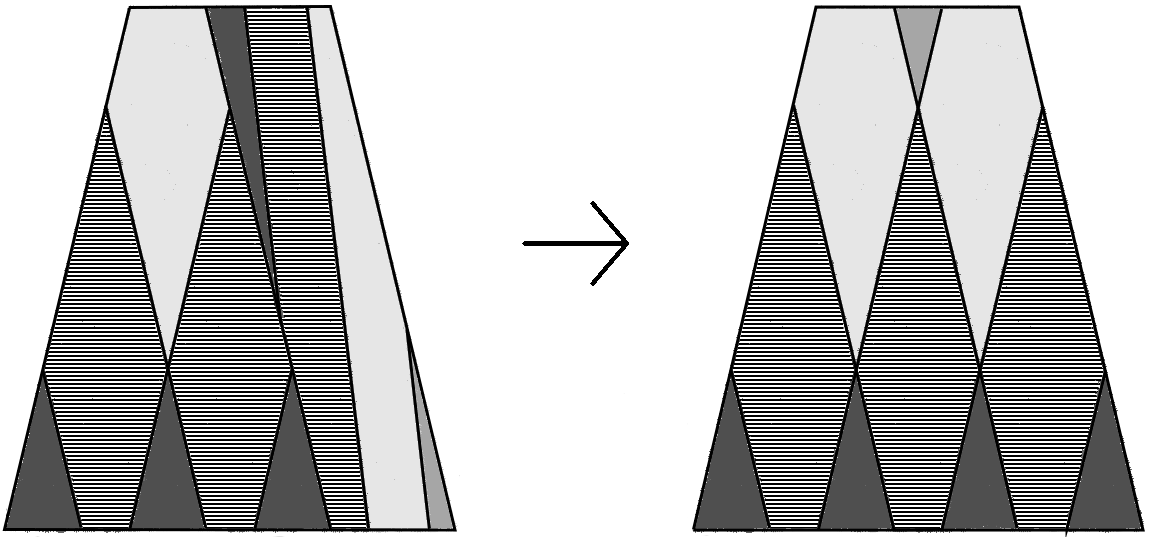}
\caption{Representation of the first return time map of $R'$ from $D_2'$ into itself.}
\label{fig:imm}
\end{figure}
\noindent
We put $\overline{\psi}$ the application from $D_2'$ into $\mathbb R^2$
defined by $\overline{\psi}(x,y)=(\phi^2 x,y)$. 
It is not hard to see that $\overline{\psi}(D_2')$ is a fundamental domain for the torus.
To conclude the proof, we have to verify that for any $(x,y)\in \overline{\psi}(D_2')$, for any $n\in\{0,1,2,3\}$,
we have $\psi \circ {R_1'}^n \circ R_2' \ (x/\phi^2,y)=T_\phi(x,y)+(n-2,0)$.
$$
\begin{array}{clll}
\psi \circ {R_1'}^n \circ R_2' \ \begin{pmatrix} x/\phi^2 \\ y  \end{pmatrix}
&=& \psi \circ {R_1'}^n \   \begin{pmatrix} x/\phi^2+1/\phi^2-1 \\ y+ x-1/(2\phi^3) \end{pmatrix}
= \psi \   \begin{pmatrix} x/\phi^2+(n+1) /\phi^2-1 \\ y+ x-1/(2\phi^3)  \end{pmatrix}
=\begin{pmatrix} x +(n+1) -\phi^2 \\ y+ x-1/(2\phi^3) \end{pmatrix} \\
& = & \begin{pmatrix} x +(n+1) -3+1/\phi^2 \\ y+ x-1/(2\phi^3) \end{pmatrix} = T_\phi(x,y) +(n-2,0).
\end{array}
$$
These first return applications are very close to those studied by P. Arnoux and C. Mauduit in \cite{MR1390569}.
\end{proof}

\noindent
\textit{\textbf{Acknowledgements:}}
I would like to thank Xavier Bressaud, Livio Flaminio, Pascal Hubert and Serge Troubetzkoy for
useful discussions and comments.
I am also grateful to the anonymous referee 
for valuable suggestions.
\bibliography{biblio}

\begin{thebibliography}{10}

\bibitem{Adler}
R.~L. Adler.
\newblock Symbolic dynamics and {M}arkov partitions.
\newblock {\em Bull. Amer. Math. Soc. (N.S.)}, 35(1):1--56, 1998.

\bibitem{MR2035027}
L.~Ambrosio and S.~Rigot.
\newblock Optimal mass transportation in the {H}eisenberg group.
\newblock {\em J. Funct. Anal.}, 208(2):261--301, 2004.

\bibitem{PierreX}
P.~Arnoux, J.~Bernat, and X.~Bressaud.
\newblock Geometric models for substitution.
\newblock {\em To appear in: Experimental Mathematics}, 2010.

\bibitem{MR1390569}
P.~Arnoux and C.~Mauduit.
\newblock Complexit\'e de suites engendr\'ees par des r\'ecurrences
  unipotentes.
\newblock {\em Acta Arith.}, 76(1):85--97, 1996.

\bibitem{PierreA}
P.~Arnoux and A.~Siegel.
\newblock Dynamique du nombre d'or.
\newblock {\em To appear in: Actes de l'université d'été de Bordeaux}, 2004.

\bibitem{MR0167569}
L.~Auslander, L.~Green, and F.~Hahn.
\newblock {\em Flows on homogeneous spaces}.
\newblock With the assistance of L. Markus and W. Massey, and an appendix by L.
  Greenberg. Annals of Mathematics Studies, No. 53. Princeton University Press,
  Princeton, N.J., 1963.

\bibitem{MR1879664}
N.~Chekhova, P.~Hubert, and A.~Messaoudi.
\newblock Propri\'et\'es combinatoires, ergodiques et arithm\'etiques de la
  substitution de {T}ribonacci.
\newblock {\em J. Th\'eor. Nombres Bordeaux}, 13(2):371--394, 2001.

\bibitem{MR2218767}
L.~Flaminio and G.~Forni.
\newblock Equidistribution of nilflows and applications to theta sums.
\newblock {\em Ergodic Theory Dynam. Systems}, 26(2):409--433, 2006.

\bibitem{MR1970385}
Pytheas Fogg.
\newblock {\em Substitutions in dynamics, arithmetics and combinatorics},
  volume 1794 of {\em Lecture Notes in Mathematics}.
\newblock Springer-Verlag, Berlin, 2002.
\newblock Edited by V. Berth{\'e}, S. Ferenczi, C. Mauduit and A. Siegel.

\bibitem{MR0133429}
H.~Furstenberg.
\newblock Strict ergodicity and transformation of the torus.
\newblock {\em Amer. J. Math.}, 83:573--601, 1961.

\bibitem{MR1326950}
G.~Gelbrich.
\newblock Self-similar periodic tilings on the {H}eisenberg group.
\newblock {\em J. Lie Theory}, 4(1):31--37, 1994.

\bibitem{MR1106594}
M.~Goze and P.~Piu.
\newblock Classification des m\'etriques invariantes \`a\ gauche sur le groupe
  de {H}eisenberg.
\newblock {\em Rend. Circ. Mat. Palermo (2)}, 39(2):299--306, 1990.

\bibitem{MR0126504}
L.~W. Green.
\newblock Spectra of nilflows.
\newblock {\em Bull. Amer. Math. Soc.}, 67:414--415, 1961.

\bibitem{LEE}
J.~R. Lee and A.~Naor.
\newblock $l_p$ metrics on the heisenberg group and the goemans-linial
  conjecture.

\bibitem{MR1116647}
E.~Lesigne.
\newblock Sur une nil-vari\'et\'e, les parties minimales associ\'ees \`a une
  translation sont uniquement ergodiques.
\newblock {\em Ergodic Theory Dynam. Systems}, 11(2):379--391, 1991.

\bibitem{MR2478814}
P.~Pansu.
\newblock Plongements quasiisom\'etriques du groupe de {H}eisenberg dans
  {$L^p$}, d'apr\`es {C}heeger, {K}leiner, {L}ee, {N}aor.
\newblock In {\em Actes du {S}\'eminaire de {T}h\'eorie {S}pectrale et
  {G}\'eom\'etrie. {V}ol. 25. {A}nn\'ee 2006--2007}, volume~25 of {\em S\'emin.
  Th\'eor. Spectr. G\'eom.}, pages 159--176. Univ. Grenoble I, Saint, 2008.

\bibitem{MR924156}
M.~Queff{\'e}lec.
\newblock {\em Substitution dynamical systems---spectral analysis}, volume 1294
  of {\em Lecture Notes in Mathematics}.
\newblock Springer-Verlag, Berlin, 1987.

\end{thebibliography}

\end{document}